\documentclass[11p]{amsart}
\usepackage{mathtools}
\usepackage{amssymb}
\usepackage{amsfonts}
\usepackage{amsthm}
\usepackage{mathrsfs}
\usepackage{polynom}
\usepackage{mathbbol}
\usepackage{latexsym,amscd}
\usepackage[all]{xy}
\usepackage{color}
\usepackage[colorlinks=true]{hyperref}
\usepackage{xcolor}
\hypersetup{citecolor=blue, urlcolor=red, colorlinks=false, linkcolor=green, citecolor=blue, filecolor=magenta}

\newtheorem{thm}{Theorem}
\newtheorem{prop}{Proposition}
\newtheorem{lem}{Lemma}
\numberwithin{equation}{section}
\numberwithin{prop}{section}
\numberwithin{lem}{section}
\numberwithin{thm}{section}

\numberwithin{cor}{section}

\theoremstyle{definition}

\numberwithin{defn}{section}

\usepackage{mathtools}
\usepackage{tikz}
\usetikzlibrary{chains}
\usepackage{graphics}
\usepackage{epic,eepic}

\newtheorem{rem}{Remark}
\numberwithin{rem}{section}

\def \<{\langle}
\def \>{\rangle}

\def \a{\alpha }

\def \b{\beta }

\newcommand{\bea}{\begin{eqnarray}}
\newcommand{\eea}{\end{eqnarray}}
\newcommand{\be}{\begin {equation}}
\newcommand{\ee}{\end{equation}}
\newcommand{\g}{\mathfrak{g}}

\newcommand{\h}{\mathfrak{h}}
\newcommand{\wt}{{\rm {wt} }   }

\newcommand{\hh}{\hat {\frak h} }

\newcommand{\Z}{\Bbb Z}

\newcommand{\xbet}[1]{x^{\hnu}_{\beta}\left(#1\right)}

\newcommand{\xa}[2]{x^{\hat{\nu}}_{\alpha_{#1}}\left(#2\right)}

\newcommand{\hnu}{\hat{\nu}}
\newcommand{\n}{\mathfrak{n}}

\newcommand{\ta}[1]{\tau_{\gamma_{#1},\theta_{#1}}}

\newcommand{\un}[0]{U\left(\overline{\mathfrak{n}}[\hnu]\right)}

\newcommand{\CC}[0]{\mathbb{C}}
\newcommand{\ZZ}[0]{\mathbb{Z}}
\newcommand{\QQ}[0]{\mathbb{Q}}

\newcommand\scalemath[2]{\scalebox{#1}{\mbox{\ensuremath{\displaystyle #2}}}}

\tikzset{node distance=2em, ch/.style={circle,draw,on chain,inner sep=2pt},chj/.style={ch,join},every path/.style={shorten >=4pt,shorten <=4pt},line width=1pt,baseline=-1ex}

\let\dlabel=\alabel

\newcommand{\dnode}[2][chj]{%
\node[#1,label={below:\dlabel{#2}}] {};
}

\newcommand{\dnodebr}[1]{%
\node[chj,label={below right:\dlabel{#1}}] {};
}

\newcommand{\dydots}{%
\node[chj,draw=none,inner sep=1pt] {\dots};
}

\begin{document}
 
\title[]{Combinatorial bases of principal subspaces of modules for twisted affine Lie algebras of type $A_{2l-1}^{(2)}, D_l^{(2)}, E_6^{(2)}$ and $D_4^{(3)}$ }  
 
\author{Marijana Butorac and Christopher Sadowski}



\begin{abstract} 
We construct combinatorial bases of principal subspaces of standard modules of level $k \geq 1$ with highest weight $k\Lambda_0$ for the twisted affine Lie algebras of type $A_{2l-1}^{(2)}$, $D_l^{(2)}$, $E_6^{(2)}$ and $D_4^{(3)}$. Using these bases we directly calculate characters of principal subspaces.
\end{abstract}

\maketitle
\section{Introduction}
The study of principal subspaces of standard modules for untwisted affine Lie algebras was initiated in \cite{FS1}--\cite{FS2} by Feigin and Stoyanovsky and was later extended by Georgiev in \cite{G}. Motivated by the work of J. Lepowsky and M. Primc in \cite{LP} and extending the earlier work of Feigin and Stoyanovsky, Georgiev constructed bases of principal subspaces of certain standard modules of the affine Lie algebra of type $A_l^{(1)}$.  These bases were described by using  certain coefficients of vertex operators, which are called quasi-particles. From quasi-particle bases, they directly obtained the characters (i.e. multigraded dimensions) of principal subspaces. The work of Feigin and Stoyanovsky and Georgiev has since been extended in many ways by other authors (cf. \cite{AKS}, \cite{Ba}, \cite{BPT}, \cite{Bu1}--\cite{Bu3}, \cite{FFJMM},  \cite{J1}--\cite{J2}, \cite{JP}, \cite{Kan}, \cite{Ka1}--\cite{Ka2}, \cite{Ko1}--\cite{Ko3}, \cite{MPe}, \cite{P}, \cite{T1}--\cite{T3}, and many others). 

The study of principal subspaces of basic modules for twisted affine Lie algebras was initiated in \cite{CalLM4}, where a general setting was given and the principal subspace of the basic $A_2^{(2)}$-module was studied. This work was later extended in \cite{CalMPe} and \cite {PS1}--\cite{PS2} to study the principal subspaces of the basic modules for all the twisted affine Lie algebras, and in \cite{PSW} to a certain lattice setting. In each of these works the authors, using certain ideas from the untwisted affine Lie algebra setting found in \cite{CLM1}--\cite{CLM2} and \cite{CalLM1}--\cite{CalLM3} (see also \cite{Cal1}--\cite{Cal2}, \cite{S1}--\cite{S2}), showed that the principal subspaces under consideration had certain presentations (i.e. could be defined in terms of certain generators and relations). Using these presentations, the authors constructed exact sequences among the principal subspaces and in this way obtained recursions satisfied by the characters of principal subspaces. Solving these recursions yields the characters of the principal subspaces of the basic modules for the twisted affine Lie algebras in each work.

The aim of this work is to determine the characters of the principal subspaces of the level $k \geq 1$ standard module $L^{\hnu}(k\Lambda_0)$, which we denote by $W^T_{L_k}$, for the twisted affine Lie algebras of type $A_{2l-1}^{(2)}$, for $l \geq 2$, $D_l^{(2)}$, for $l \geq 4$, $E_6^{(2)}$ and $D_4^{(3)}$, extending certain results found in \cite{PS1}--\cite{PS2}. The approach we use, however, is different than the approach found in \cite{PS1}--\cite{PS2}. By using vertex operator techniques we construct combinatorial bases of principal subspaces which are twisted analogues to those found in \cite{G}, and from which we obtain the characters of principal subspaces. In our proofs, we use level $k$ analogues of certain maps originally developed and used in \cite{CalLM4}, \cite{CalMPe}, and \cite{PS1}--\cite{PS2} analogously to how they were used in \cite{G}. We note importantly that, in the untwisted setting, certain modes of intertwining operators play an important role in the proofs of linear independence of these bases. In our twisted affine Lie algebra setting, we instead use other maps developed for the twisted setting in \cite{CalLM4}.

More specifically, in this paper, following \cite{G} and using certain results in \cite{Li}, we construct bases using the coefficients
\begin{equation*}
x^{\hnu}_{r\alpha_i}(m)=\text{Res}_z \{z^{m+r-1} x^{\hnu}_{r\alpha_i}(z)\},
\end{equation*}
of the twisted vertex operators 
\begin{equation*} 
x^{\hnu}_{r\alpha_i}(z)=Y^{\hnu}(x_{\alpha_i}(-1)^r {\bf 1} ,  z),
\end{equation*}
which we, in complete analogy with the untwisted case, call twisted quasi-particles of color $i$, charge $r$ and energy $-m$.  Similar to the untwisted case, (see \cite{Bu1}--\cite{Bu3}, \cite{JP}), first we prove certain relations for twisted quasi-particles of the form $x^{\hnu}_{r\alpha_i}(m)x^{\hnu}_{r'\alpha_i}(m')$ for $r\leq r'$ and $x^{\hnu}_{r\alpha_i}(m)x^{\hnu}_{r'\alpha_j}(m')$, where $1 \leq r, r' \leq k$, which we call relations among twisted quasi-particles. With these relations, along with the relations $x^{\hnu}_{(k+1)\alpha_i}(z)=0$, we build twisted quasi-particle spanning sets of the principal subspaces $W^T_{L_k}$. The resulting bases are analogous to the quasi-particle bases of principal subspaces in the case of untwisted affine Lie algebras of type $ADE$ in the sense that energies of twisted quasi-particles in the twisted quasi-particle spanning sets satisfy similar difference conditions, which are generalizations of difference two conditions found in \cite{FS1}--\cite{FS2} and \cite{G}.
As in the untwisted case found in \cite{G}, in the proof of linear independence of our spanning sets we consider the principal subspace as a subspace of tensor product of $k$ principal subspaces of basic modules. This enables us to use the above-mentioned maps obtained from the construction of level one twisted modules for lattice vertex operator algebras from \cite{CalLM4} and \cite {PS1}--\cite{PS2}.  Finally, we note that in this paper we do not consider the case of the principal subspaces $W^T_{L_k}$ for the twisted affine Lie algebra of type $A_{2l}^{(2)}$, which will be considered in future work. 

Our main result in this work is as follows: denote by $\mathrm{ch} \  W^T_{L_k}$ the character of the principal subspace $W^T_{L_k}$. The characters of the principal subspaces are:
\begin{thm} We have for $A_{2l-1}$:
\begin{align}\nonumber 
&\mathrm{ch} \  W^T_{L_k}&\\
\nonumber
= &\sum_{\substack{r_{1}^{(1)}\geq \cdots\geq r_{1}^{(k)}\geq 0\\ \cdots \\r_{l}^{(1)}\geq \cdots\geq r_{l}^{(k)}\geq 
0}}
\frac{q^{\frac{1}{2}\sum_{i=1}^{l-1}\sum_{s=1}^{k}r_i^{(s)^2}+\sum_{s=1}^{k}r_l^{(s)^2}-\frac{1}{2}\sum_{i=2}^{l-1}\sum_{s=1}^{k}r_{i-1}^{(s)}r_{i}^{(s)}-\sum_{s=1}^{k}r_{l-1}^{(s)}r_{l}^{(s)}}}
{\prod_{i=1}^{l-1}\left((q^{\frac{1}{2}};q^{\frac{1}{2}})_{r^{(1)}_{i}-r^{(2)}_{i}}\cdots (q^{\frac{1}{2}};q^{\frac{1}{2}})_{r^{(k)}_{i}}\right) (q)_{r^{(1)}_{l}-r^{(2)}_{l}}\cdots (q)_{r^{(k)}_{l}}}y^{r^{(1)}_1+\cdots +r^{(k)}_1}_{1} \cdots y^{r^{(1)}_l+\cdots +r^{(k)}_l}_{l}&
\end{align}
for $D_l$ when $v=2$:
 \begin{align}\nonumber
&\mathrm{ch} \  W^T_{L_k}&\\
\nonumber
= &\sum_{\substack{r_{1}^{(1)}\geq \cdots\geq r_{1}^{(k)}\geq 0\\ \cdots \\r_{l-1}^{(1)}\geq \cdots\geq r_{l-1}^{(k)}\geq 
0}}
\frac{q^{\sum_{i=1}^{l-2}\sum_{s=1}^{k}r_i^{(s)^2}+\frac{1}{2}\sum_{s=1}^{k}r_{l-1}^{(s)^2}-\sum_{i=2}^{l-1}\sum_{s=1}^{k}r_{i-1}^{(s)}r_{i}^{(s)}}}
{\prod_{i=1}^{l-2}\left((q)_{r^{(1)}_{i}-r^{(2)}_{i}}\cdots  (q)_{r^{(k)}_{i}}\right)(q^{\frac{1}{2}};q^{\frac{1}{2}})_{r^{(1)}_{l-1}-r^{(2)}_{l-1}}\cdots (q^{\frac{1}{2}};q^{\frac{1}{2}})_{r^{(k)}_{l-1}}}y^{r^{(1)}_1+\cdots +r^{(k)}_1}_{1} \cdots y^{r^{(1)}_{l-1}+\cdots +r^{(k)}_{l-1}}_{l-1}&
\end{align}
 for $E_6$:
 \begin{align}\nonumber
&\mathrm{ch} \  W^T_{L_k}&\\
\nonumber
= &\sum_{\substack{r_{1}^{(1)}\geq \cdots\geq r_{1}^{(k)}\geq 0\\\substack{ \ldots \\ r_{4}^{(1)}\geq \cdots\geq r_{4}^{(k)}\geq 
0}}}
\frac{q^{\frac{1}{2}\sum_{i=1}^{2}\sum_{s=1}^{k}r_i^{(s)^2}+\sum_{i=3}^{4}\sum_{s=1}^{k}r_i^{(s)^2}-\sum_{s=1}^{k}(r_{1}^{(s)}r_{2}^{(s)}+r_{2}^{(s)}r_{3}^{(s)}+r_{3}^{(s)}r_{4}^{(s)})}}
{\prod_{i=1,2}\left((q^{\frac{1}{2}};q^{\frac{1}{2}})_{r^{(1)}_{i}-r^{(2)}_{i}}\cdots (q^{\frac{1}{2}};q^{\frac{1}{2}})_{r^{(k)}_{i}}\right)\left(\prod_{i=3,4} (q)_{r^{(1)}_{i}-r^{(2)}_{i}}\cdots (q)_{r^{(k)}_{i}}\right)}y^{r^{(1)}_1+\cdots +r^{(k)}_1}_{1} \cdots y^{r^{(1)}_4+\cdots +r^{(k)}_4}_{4}&
\end{align}
and for $D_4$ when $v=3$:
\begin{align}\nonumber 
&\mathrm{ch} \  W^T_{L_k}&\\
\nonumber
= &\sum_{\substack{r_{1}^{(1)}\geq \cdots\geq r_{1}^{(k)}\geq 0\\\substack{  r_{2}^{(1)}\geq \cdots\geq r_{2}^{(k)}\geq 
0}}}
\frac{q^{\frac{1}{3}\sum_{s=1}^{k}r_1^{(s)^2}+ \sum_{s=1}^{k}r_2^{(s)^2}-\sum_{s=1}^{k}r_{1}^{(s)}r_{2}^{(s)}}}
{(q^{\frac{1}{3}};q^{\frac{1}{3}})_{r^{(1)}_{1}-r^{(2)}_{1}}\cdots (q^{\frac{1}{3}};q^{\frac{1}{3}})_{r^{(k)}_{1}}  (q)_{r^{(1)}_{2}-r^{(2)}_{2}}\cdots (q)_{r^{(k)}_{2}}}y^{r^{(1)}_1+\cdots +r^{(k)}_1}_{1}  y^{r^{(1)}_2+\cdots +r^{(k)}_2}_{2}.&
\end{align}
\end{thm}

\section{Preliminaries}
In this section, we very closely follow the setting developed in \cite{PS1}--\cite{PS2} (cf. also \cite{L1} and \cite{CalLM4}), and recall many details from these works.

Let $\mathfrak{g}$ be a finite dimensional simple Lie algebra of type $A_{2l-1}, D_l,$ or $E_6$, with root lattice 
\begin{equation*}
L = \mathbb{Z} \alpha_1 \oplus \dots \oplus \mathbb{Z} \alpha_D,
\end{equation*}
where $D$ is the rank of $\mathfrak{g}$, with its standard nondegenerate symmetric bilinear form $\langle \cdot, \cdot \rangle$. Also, let
\begin{equation*}
\mathfrak{h} = L \otimes_\mathbb{Z} \mathbb{C}.
\end{equation*}
We take the following labelings of the Dynkin diagrams of our Lie algebras:\\
{\bf Type $A_{2l-1}$:}
\begin{center}\begin{tikzpicture}[start chain]
\dnode{1}
\dnode{2}
\dydots
\dnode{l-1}
\dnode{l}
\dnode{l+1}
\dydots
\dnode{2l-2}
\dnode{2l-1}
\end{tikzpicture}\end{center}
{\bf Type $D_l$:}
\begin{center}
\begin{tikzpicture}
\begin{scope}[start chain]
\dnode{1}
\dnode{2}
\node[chj,draw=none] {\dots};
\dnode{l-2}
\dnode{l-1}
\end{scope}
\begin{scope}[start chain=br going above]
\chainin(chain-4);
\dnodebr{l}
\end{scope}
\end{tikzpicture}\end{center}
{\bf Type $E_6$:}
\begin{center}\begin{tikzpicture}
\begin{scope}[start chain]
\dnode{1}
\dnode{2}
\dnode{3}
\dnode{5}
\dnode{6}
\end{scope}
\begin{scope}[start chain=br going above]
\chainin (chain-3);
\dnodebr{4}
\end{scope}
\end{tikzpicture}\end{center}
In the case of $D_4$, we use the labeling:
\begin{center}
\begin{tikzpicture}
\begin{scope}[start chain]
\dnode{1}
\dnode{2}
\dnode{3}
\end{scope}
\begin{scope}[start chain=br going above]
\chainin(chain-2);
\dnodebr{4}
\end{scope}
\end{tikzpicture}\end{center}
\begin{rem}
We note here that in \cite{PS2}, the labeling used for $E_6$ was:
\begin{center}\begin{tikzpicture}
\begin{scope}[start chain]
\foreach \dyni in {1,...,5} {
\dnode{\dyni}
}
\end{scope}
\begin{scope}[start chain=br going above]
\chainin (chain-3);
\dnodebr{6}
\end{scope}
\end{tikzpicture}\end{center}
and in \cite{PS1}, the labeling used for $D_4$ was:
\begin{center}
\begin{tikzpicture}
\begin{scope}[start chain]
\dnode{2}
\dnode{1}
\dnode{3}
\end{scope}
\begin{scope}[start chain=br going above]
\chainin(chain-2);
\dnodebr{4}
\end{scope}
\end{tikzpicture}\end{center}
and we change the labeling in this work for notational simplicity. Later in the work, we will see that the roles of operators corresponding to $\alpha_4$ and $\alpha_6$  in the case of $E_6$ and the roles of operators corresponding to $\alpha_1$ and $\alpha_2$  in the case of $D_4$ will be swapped compared to their counterparts found in  \cite{PS2} and \cite{PS1}.
\end{rem}
\subsection{Dynkin diagram automorphisms}
Let $\nu$ be a Dynkin diagram automorphism of $\mathfrak{g}$ of order $v$, extended to all of $\mathfrak{h}$. In the case that $v=2$, we let $\eta=-1$ be a primitive second root of unity and set $\eta_0 = \eta$, and in the case that $v=3$ we let $\eta$ be a cube root of unity and set $\eta_0 = -\eta$. Following \cite{CalLM4} and \cite{PS1}--\cite{PS2}, we consider two central extensions of $L$ by the group $\langle\eta_0\rangle$, denoted by $\hat{L}$ and $\hat{L}_\nu$, with commutator maps $C_0$ and $C$ and associated normalized $2$-cocycles $\epsilon_{C_0}$ and $\epsilon_{C}$, respectively:
\begin{equation*}
1\longrightarrow\left<\eta_0\right>\longrightarrow\hat{L}\overset{\overline{~~}}{\longrightarrow} L\longrightarrow 1
\end{equation*}
and
\begin{equation*}
1\longrightarrow\left<\eta_0\right>\longrightarrow\hat{L}_\nu\overset{\overline{~~}}{\longrightarrow} L\longrightarrow 1
\end{equation*}
 We define commutator maps $C_0$ and $C$ by
\begin{equation*}\begin{aligned}
C_0:L\times L&\to\CC^{\times}\\
(\alpha,\b)&\mapsto(-1)^{\left<\a,\b\right>}
\end{aligned}\end{equation*}
and 
\begin{equation*}
C(\a,\b)=\prod_{j=0}^{v-1}(-\eta^j)^{\left<\nu^j\a,\b\right>}.
\end{equation*}
Following \cite{L1} and \cite{CalLM4}, we let 
\begin{align*}
e:L&\to\hat{L}\\
\a&\mapsto e_{\a}\end{align*}
be a normalized section of $\hat{L}$ so that 
\begin{equation*}e_0=1
\end{equation*}
and
\begin{equation*}
\overline{e_{\a}}=\a~\text{ for all }~\a\in L,\end{equation*}
 satisfying
\begin{equation*}
 e_{\a}e_{\b}=\epsilon_{C_0}(\a,\b)e_{\a+\b}~ \text{ for all }~\a,\b\in L.\end{equation*}
We choose our $2$-cocycle to be 
\begin{equation*}
\epsilon_{C_0}(\a_i,\a_j) = \left\{
     \begin{array}{lr}
       1 & \text{ if } i\leq j \\
       (-1)^{\left<\a_i,\a_j\right>} & \text{ if } i>j
     \end{array}
   \right.
\end{equation*}
 The $2$-cocycles $\epsilon_C$ and $\epsilon_{C_0}$ are related by (see Equation 2.21 of \cite{CalLM4})
\begin{equation*}
\epsilon_{C_0}(\a,\b)=\prod_{-\frac{k}{2}<j<0}\left(-\eta^{-j}\right)^{\left<\nu^{-j}\a,\b\right>}\epsilon_{C}(\a,\b).
\end{equation*}

We now lift the isometry $\nu$ of $L$ to an automorphism $\hnu$ of $\hat{L}$ such that 
   \begin{equation*}
   \overline{\hnu a}=\nu\overline{a} \hspace{0.4in} \text{ for } \hspace{0.4in} a\in\hat{L}.
   \end{equation*}
and choose $\hnu$ so that 
   \begin{equation*}
   \hnu a=a \hspace{0.4in} \text{ if } \hspace{0.4in} \nu\overline{a}=\overline{a},
   \end{equation*}
   and thus $\hnu^2=1$ if $\nu$ has order $2$ and $\hnu^3 = 1$ if $\nu$ has order $3$.  Indeed, set 
   \begin{equation*}
   \hnu e_{\a}=\psi(\a)e_{\nu\a} \end{equation*}
   where $\psi:L\to \left<\eta\right>$ is defined by
   \begin{equation*}
\psi(\alpha)=\begin{cases} 
      \epsilon_{C_0}(\a,\a) & \text{if $L$ is type $A_{2l-1}$} \\
      1 & \text{if $L$ is type $D_{l}$ and the order of $\nu$ is $2$} \\
      (-1)^{r_3r_4}\epsilon_{C_0}(\a,\a) & \text{if $L$ is type $E_6$ and $\a=\sum_{i=1}^6r_i\a_i$}\\
      (-1)^{r_2r_3}\epsilon_{C_0}(\a,\a) & \text{ if $L$ is of type $D_4$,  $\a=\sum_{i=1}^4r_i\a_i$ and the order of $\nu$ is $3$}.
   \end{cases}
\end{equation*}
From \cite{PS1}--\cite{PS2}, we have that:
\begin{equation*}
\epsilon_{C_0}(\nu\a,\nu\beta)=\begin{cases}
\epsilon_{C_0}(\beta,\a)& \text{if $L$ is type $A_{2l-1}$} \\
      \epsilon_{C_0}(\a,\beta) & \text{if $L$ is type $D_{l}$} \\
      (-1)^{r_4s_3+r_3s_4}\epsilon_{C_0}(\beta,\a) & \text{if $L$ is type $E_6$, $\a=\sum_{i=1}^6r_i\a_i$ and $\b=\sum_{i=1}^6s_i\a_i$}\\
       (-1)^{r_2s_3+r_3s_2}\epsilon_{C_0}(\beta,\a) & \text{if $L$ is of type $D_4$, $\a=\sum_{i=1}^3r_i\a_i$ and $\b=\sum_{i=1}^3s_i\a_i$. } 
   \end{cases}
\end{equation*}
As in \cite{PS1}--\cite{PS2}, we have that 
\begin{equation*}
 \hnu(e_{\alpha_i}) = e_{\nu \alpha_i}\\
\end{equation*}
for each simple root $\alpha_i$.

\subsection{The lattice vertex operator $V_L$ and its twisted module $V_L^T$}
We assume that the reader is familiar with the construction of the lattice vertex operator algebra $V_L$ (cf. \cite{FLM2} and \cite{LL}), and recall some important details of this construction. In particular, we follow Section 2 of \cite{CalLM4}.

We view $\mathfrak{h}$ as an abelian Lie algebra, and let
\begin{equation*}
\hat{\h}=\h\otimes \CC[t,t^{-1}] \oplus \mathbb{C}c
\end{equation*}
with the usual bracket, and let
\begin{equation*}
\hat{\h}^{-}=\h\otimes t^{-1}\CC[t^{-1}].
\end{equation*}
We have that
\begin{equation*}
V_L \cong S(\hat{\h}^{-} ) \otimes \mathbb{C}[L]
\end{equation*}
linearly. We extend $\hnu$ to an automorphism of $V_L$, which we also call $\hnu$, by $\hnu = \nu \otimes \hnu$.

Let
\begin{equation*}
\mathfrak{h}_{(m)} = \{ x \in \mathfrak{h} \  | \ \nu(x) = \eta^m x \}.
\end{equation*}
We have that
\begin{equation*}
\mathfrak{h} = \coprod_{m \in \mathbb{Z}/v\mathbb{Z}} \mathfrak{h}_{(m)}.
\end{equation*}
Let $P_0$ be the projection of $\mathfrak{h}$ onto $\mathfrak{h}_{(0)}$.
We form the twisted affine Lie algebra
\begin{equation*}
\hat{\mathfrak{h}}[\nu] = \coprod_{m \in \mathbb{Z}} \mathfrak{h}_{(m)} \otimes t^{m/v}
\end{equation*}
where
\begin{equation*}
[\alpha \otimes t^m, \beta \otimes t^n] = \langle \alpha, \beta \rangle m\delta_{m+n,0}c
\end{equation*}
for $m,n \in \frac{1}{v} \mathbb{Z}$ and $\alpha \in \mathfrak{h}_{(4m)}$ and $\beta \in \mathfrak{h}_{(4n)}$
and $c$ is central. The Lie algebra $\hat{\mathfrak{h}}[\nu]$ is $\frac{1}{v}\mathbb{Z}$-graded by weights:
\begin{equation*}
\text{wt}(\alpha \otimes t^m) = -m \ \ \text{and} \ \ \text{wt}(c) = 0.
\end{equation*}
Define the Heisenberg subalgebra
\begin{equation*}
\hh[\nu]_{\frac{1}{v}\ZZ}=\prod_{\substack{m\in\frac{1}{v}\ZZ \\ m\neq 0}}\h_{(km)}\otimes t^m\oplus\CC c
\end{equation*}
of $\hh[\nu]$,
the subalgebras
\begin{equation*}\hh[\nu]^{\pm}=\prod_{\substack{m\in\frac{1}{v}\ZZ \\ \pm m>0}}\h_{(vm)}\otimes t^m
\end{equation*}
of $\hh[\nu]_{\frac{1}{v}\ZZ}$,
and the induced module
\begin{equation*}
S[\nu]=U\left(\hh[\nu]\right)\otimes_{\prod_{m\geq 0}\h_{(vm)}\otimes t^m\oplus\CC c}\CC\cong S\left(\hh[\nu]^{-}\right),
\end{equation*}
which is $\QQ$-graded such that 
\begin{equation*}
\text{wt}(1)=\frac{1}{4v^2}\sum_{j=1}^{v-1}j(v-j)\text{dim}\mathfrak{h}_{(j)}.\end{equation*}

Following \cite{L1} and \cite{CalLM4}, we set
\begin{equation*}
N = (1-P_0)\mathfrak{h} \cap L.
\end{equation*}
In the case that $v=2$, we have that
\begin{equation*}
N = \coprod_{i=1}^D \mathbb{Z} (\alpha_i - \nu \alpha_i)
\end{equation*}
and when $v=3$ we have that:
\begin{equation*}
N = \{ r_1 \a_1 +r_3 \a_3 + r_4 \a_4 \in L \  |  \  r_1 + r_3 + r_4 = 0 \}.
\end{equation*}

 Using Proposition 6.2 of \cite{L1}, let $\CC_{\tau}$ denote the one dimensional $\hat{N}$-module $\CC$ with character $\tau$ and write
\begin{equation*}
 T=\CC_{\tau}.\end{equation*} Consider the induced $\hat{L}_{\nu}$-module 
\begin{equation*}
U_{T}=\CC[\hat{L}_{\nu}]\otimes_{\CC[\hat{N}]}T\cong \CC[L/N],\end{equation*}
which is graded by weights and on which $\hat{L}_{\nu}$, $\h_{(0)}$, and $z^h$ for $h\in\h_{(0)}$ all naturally act. Set
\begin{equation*}
V_{L}^{T}=S[\nu]\otimes U_T\cong S\left(\hat{\h}[\nu]^{-}\right)\otimes \CC[L/N],\end{equation*}
which is naturally acted upon by $\hat{L}_{\nu}$, $\hat{\h}_{\frac{1}{v}\ZZ}$, $\h_{(0)}$, and $z^h$ for $h\in\h$.

For each $\a\in\h$ and $m\in\frac{1}{v}\ZZ$ define the operators on $V_L^T$
\begin{equation*}
\a_{(vm)}\otimes t^m\mapsto \a^{\hnu}(m)\end{equation*}
and set
\begin{equation*}
\a^{\hnu}(z)=\sum_{m\in\frac{1}{v}\ZZ}\a^{\hnu}(m)z^{-m-1}.\end{equation*}
Of most importance will be the $\hnu$-twisted vertex operators acting on $V_L^T$ for each $e_{\a}\in\hat{L}$
\begin{equation*}
Y^{\hnu}(\iota(e_{\a}),z)=v^{-\frac{\left<\a,\a\right>}{2}}\sigma(\alpha)E^{-}(-\a,z)E^{+}(-\a,z)e_{\alpha}z^{\alpha_{(0)}+\frac{\left<\alpha_{(0)},\alpha_{(0)}\right>}{2}-\frac{\left<\alpha,\alpha\right>}{2}},\end{equation*}
as defined in \cite{L1}, where
\begin{equation*}
E^{\pm}(-\alpha,z)=\text{exp}\left(\sum_{m\in\pm\frac{1}{v}\mathbb{Z}_{+}}\frac{-\alpha_{(vm)}(m)}{m}z^{-m}\right),
\end{equation*}
and
\begin{equation*}
\sigma(\alpha) = 1\text{ when } v = 2
\end{equation*}
and
\begin{equation*}
\sigma(\alpha) = (1-\eta^2)^{\langle \nu \alpha, \alpha \rangle} \text{ when } v = 3
\end{equation*}
for $\a\in\h$. For $m\in\frac{1}{v}$ and $\a\in L$ define the component operators $\xa{}{m}$ by 
\begin{equation*}\label{VertexOperators}
Y^{\hnu}(\iota(e_{\alpha}),z)=\sum_{m\in\frac{1}{v}\mathbb{Z}}\xa{}{m}z^{-m-\frac{\left<\alpha,\alpha\right>}{2}}=\xa{}{z}.
\end{equation*}
We note here that $V_L^T$ is a $\hnu$-twisted module for $V_L$, and in particular it satisfies the twisted Jacobi identity:
\begin{multline} \label{Jacobi}
x^{-1}_0\delta\left(\frac{x_1-x_2}{x_0}\right)
Y^{\hat{\nu}}(u,x_1)Y^{\hat{\nu}}(v,x_2)-x^{-1}_0
\delta\left(\frac{x_2-x_1}{-x_0}\right) Y^{\hat{\nu}}
(v,x_2)Y^{\hat{\nu}} (u,x_1)\\  
= x_2^{-1}\frac{1}{v}\sum_{j\in \Z /v \Z}
\delta\left(\eta^j\frac{(x_1-x_0)^{1/v}}{x_2^{1/v}}\right)Y^{\hat{\nu}} 
(Y(\hat{\nu}^j
u,x_0)v,x_2) 
\end{multline}
for $u, v \in V_L$.

\subsection{Twisted affine Lie algebras}
We now construct the twisted affine Lie algebras of type $A_{2l-1}^{(2)}, D_l^{(2)}, E_6^{(2)},$ and $D_4^{(3)}$, and give them an action on $V_L^T$. Define the vector space
\begin{equation*}
\g=\h\oplus\coprod_{\a\in\Delta}\mathbb{C} x_{\a},
\end{equation*}
where $\{x_{\a}\}$ is a set of symbols, and $\Delta$ is the set of roots corresponding to $L$. 

We give the vector space $\g$ the structure of a Lie algebra via the bracket defined 
\begin{equation*}
[h,x_{\a}]=\left<h,\a\right>x_{\a},~~[\h,\h]=0\end{equation*} 
where $h\in\h$ and $\a\in\Delta$ and
\begin{equation*}
   [x_{\alpha},x_{\beta}] = \left\{
     \begin{array}{lr}
       \epsilon_{C_0}(\alpha,-\alpha)\alpha & \text{ if }  \alpha+\beta=0\\
        \epsilon_{C_0}(\alpha,\beta)x_{\alpha+\beta} & \text{ if }  \alpha+\beta\in\Delta\\
        0   & \text{ otherwise}.
     \end{array}
   \right.
\end{equation*}  
We note that $\g$ is a Lie algebra isomorphic to one of type $A_{2l-1}$, $D_l$, or $E_6$ depending on the choice of $L$ (cf. \cite{FLM3}). We also extend the bilinear form $\left<\cdot,\cdot\right>$ to $\g$ by
\begin{equation*} \left<h,x_{\a}\right>=\left<x_{\a},h\right>=0\end{equation*} 
and
\begin{equation*}
\left<x_{\alpha},x_{\beta}\right>=\left\{\begin{array}{lrr}
\epsilon_{C_0}(\alpha,-\alpha)     &\text{if}~~~~\alpha+\beta=&0 \\
0     &\text{if}~~~~~\alpha+\beta\neq&0
\end{array}\right.\end{equation*}

Following \cite{L1}, \cite{CalLM4}, and \cite{PS1}--\cite{PS2}, we use our extension of $\nu:L\to L$ to $\hnu:\hat{L}\to\hat{L}$ to lift the automorphism $\nu:\h\to\h$ to an automorphism $\hnu:\g\to\g$ by setting
\begin{equation*}
\hnu x_{\a}=\psi (\alpha) x_{\nu\a}
\end{equation*} for all $\a\in\Delta$. Here, we are using our particular choices of $\hnu$ (extended to $\mathbb{C}\{L \}$) and section $e$.

For $m\in\mathbb{Z}$ set
\begin{equation*}
\g_{(m)}=\{x\in\g \  | \ \hnu(x)=\eta^m x\}.
\end{equation*}
Form the $\hnu$-twisted affine Lie algebra associated to $\g$ and $\hnu$:
\begin{equation*} 
\hat{\g}[\hnu]=\coprod_{m\in\frac{1}{v}\mathbb{Z}}\g_{(vm)}\otimes t^{m}\oplus\mathbb{C}c
\end{equation*}
with
\begin{equation*}
[x\otimes t^m,y\otimes t^n]=[x,y]\otimes t^{m+n}+\left<x,y\right>m\delta_{m+n,0}c
\end{equation*}
and
\begin{equation*}
[c,\hat{\g}[\hnu]]=0,
\end{equation*}
for $m,n\in\frac{1}{v}\mathbb{Z}$, $x\in\g_{(vm)}$, and $y\in\g_{(vn)}$.
Adjoining the degree operator $d$ to  $\hat{\g}[\hnu]$, we define
\begin{equation*}
\tilde{\g}[\hnu]=\hat{\g}[\hnu]\oplus\mathbb{C}d,
\end{equation*}
where
\begin{equation*}
[d,x\otimes t^n]=n x\otimes t^n,
\end{equation*}
for $x\in\g_{(vn)}$, $n\in\frac{1}{v}\mathbb{Z}$ and $[d,c]=0$. The Lie algebra $\tilde{g}[\hnu]$ is isomorphic to $A_{2l-1}^{(2)}$, $D_l^{(2)}$, $E_6^{(2)}$, or $D_4^{(3)}$ depending on the choice of $L$ and $\nu$, and is $\frac{1}{v}\mathbb{Z}$-graded.
We give $V_L^T$ the structure of a $\hat{\g}[\hnu]$-module by:
\begin{thm}(Theorem 3.1 \cite{CalLM4}, Theorem 9.1 \cite{L1}, Theorem 3 \cite{FLM2})\label{reptheorem}
The representation of $\hat{\h}[\nu]$ on $V_L^T$ extends uniquely to a Lie algebra representation of $\hat{\g}[\hnu]$ on $V_L^T$ such that
\begin{equation*}
(x_{\alpha})_{(vm)}\otimes t^n\mapsto \xa{}{m}
\end{equation*}
for all $m\in\frac{1}{v}\mathbb{Z}$ and $\alpha\in L$. Moreover $V_L^T$ is irreducible as a $\hat{\g}[\hnu]$-module.
\end{thm}

\subsection{Gradings}
As in Section 2 of \cite{CalLM4} (also Section 6 of \cite{L1}) we have a tensor product grading on $V_L^T$ given by the action of $L^{\hnu}(0)$, where \begin{equation*} Y^{\hnu}(\omega,z)=\sum_{m\in\ZZ}L^{\hnu}(m)z^{-m -2},\end{equation*}
 which we call the {\em weight grading}. In particular, we have
\begin{equation*}\begin{aligned}
\wt(1)&=\frac{l-1}{16}, \text{ for } A_{2l-1}\\
\wt(1)&=\frac{1}{16}, \text{ for } D_l \text{ when } v=2\\
\wt(1)&=\frac{1}{8}, \text{ for } E_6,\\
\wt(1) &= \frac{1}{9}, \text{ for } D_4 \text{ when } v=3.\end{aligned}\end{equation*}
From \cite{PS1}--\cite{PS2}, we recall that
\begin{equation*}
\wt(\xa{}{m})=-m-1+\frac{1}{2}\left<\alpha,\alpha\right>
\end{equation*}
for $m \in \frac{1}{v}\mathbb{Z}$ and $\alpha \in L$.

We endow $V_L^T$ with {\em charge} gradings (see \cite{PS1}--\cite{PS2}). In particular,
In the case of $A_{2l-1}$, we have
\be \label{charge1}
\text{ch}(\xa{}{m})=\left<2\left<\a,(\lambda_{1})_{(0)}\right>,\dots,2\left<\a,(\lambda_{l-1})_{(0)}\right>,\left<\a,(\lambda_{l})_{(0)}\right>\right).\ee
In the case of $D_l$ when $v=2$ we have
\be\label{charge2}\text{ch}(\xa{}{m})=\left<\left<\a,(\lambda_{1})_{(0)}\right>,\dots,\left<\a,(\lambda_{l-2})_{(0)}\right>,2\left<\a,(\lambda_{l-1})_{(0)}\right>\right).\ee
In the $E_6$ case we have 
\be\label{charge3}
\text{ch}(\xa{}{m})=\left(2\left<\a,(\lambda_{1})_{(0)}\right>,2\left<\a,(\lambda_{2})_{(0)}\right>,\left<\a,(\lambda_{3})_{(0)}\right>,\left<\a,(\lambda_{4})_{(0)}\right>\right),\ee
and finally in the case of $D_4$ when $v=3$ we have
\be\label{charge4}
\text{ch}(\xa{}{m})=\left(3\left<\a,(\lambda_{1})_{(0)}\right>,\left<\a,(\lambda_{2})_{(0)}\right>\right).
\ee
We note that the $\lambda_i$ here are the fundamental weights of the underlying finite dimensional Lie algebra $\mathfrak{g}$, dual to the roots: 
\begin{equation*}
\langle \lambda_i, \alpha_j \rangle = \delta_{i,j}
\end{equation*}
with $1 \le i,j \le \text{rank}(\mathfrak{g})$.

\subsection{Higher levels}
The primary focus of this section so far has been the construction of $V_L^T$, the basic module for the twisted affine Lie algebra $\hat{\g}[\hnu]$. In the remainder of this work we will consider the standard $\tilde{\g}[\hnu]$-module $L^{\hnu} (k \Lambda_0)$ of level $k\geq 1$ with highest weight $k\Lambda_0$ such that $\langle k\Lambda_0, c \rangle = k$ and $\langle \Lambda_0, \mathfrak{h}_{(0)} \rangle = 0=\langle \Lambda_0, d \rangle$. We note that when $k=1$, we have $L^{\hnu} (\Lambda_0) \cong V_L^T$.

Since $L ^{\hnu}(k \Lambda_0)$ is a faithful $\hnu$-twisted $L (k \Lambda_0)$-module (see \cite{Li}), where   $L (k \Lambda_0)$ denotes level $k$ standard module of untwisted affine Lie algebra $\tilde{\g}$ which has a structure of a vertex operator algebra, from Proposition 2.10 in \cite{Li} follows that for $r \in \mathbb{N}$ and $x_{\alpha} \in \mathfrak{g}$
 \begin{equation*} 
Y^{\hnu}(x_{\alpha}(-1)^r{\bf 1}, z)=Y^{\hnu}(x_{\alpha}(-1){\bf 1}, z)^r
 \end{equation*}
is a twisted vertex operator. 

We will also use the commutator formula for twisted vertex operators
\begin{equation}\label{comform}
\left[ x_{\alpha}^{\hnu}(z_1),  x_{n'\beta}^{\hnu}(z_2) \right]=\sum_{j \geq 0} \frac{(-1)^j}{j!}\left(\frac{d}{dz_1}\right)^jz_2^{-1}\frac{1}{v} \sum_{q \in \mathbb{Z}/v\mathbb{Z}}\delta\left(\eta^{q}\frac{z_1^{\frac{1}{v}}}{z_2^{\frac{1}{v}}}\right)Y^{\hnu}(\hnu^qx_{\alpha}(j)x_{\beta}(-1)^{n'}{\bf 1}, z_2).
\end{equation} 

In this work, we realize $L ^{\hnu}(k \Lambda_0)$ as a submodule of the tensor product of $k$ copies of the basic module $V_L^T \cong L^{\hnu} (\Lambda_0)$ as follows:
\begin{equation*}
L^{\hnu} (k \Lambda_0)\cong U(\tilde{\g}[\hnu]) \cdot v_{L} \subset  V_L^T \otimes  \cdots \otimes V_L^T={V_L^T} ^{\otimes k},
\end{equation*}
where $v_{L}={\bf 1}_T \otimes {\bf 1}_T \otimes \cdots \otimes {\bf 1}_T$ is a highest weight vector of $L ^{\hnu}(k \Lambda_0)$ and where ${\bf 1}_T$ is a highest weight vector of $V_L^T$ (cf. \cite{K}). 

It is known that $V_L^{\otimes k}= V_L \otimes  \cdots \otimes V_L$ has a structure of vertex operator algebra.  If we denote by $\hnu$ the automorphism $\hnu\otimes  \cdots \otimes \hnu$ of the vertex operator algebra $V_L^{\otimes k}$, then we have $\hnu^v=1$ and one can also define vertex operators corresponding to elements $v_1\otimes \cdots \otimes v_k \in {V_L^T} ^{\otimes k}$ as tensor products of twisted vertex operators on the appropriate tensor factors $Y^{\hnu}(v_1, z) \otimes \cdots \otimes Y^{\hnu}(v_k, z)$. In this way ${V_L^T} ^{\otimes k}$ becomes an irreducible $\hnu$-twisted module for the vertex operator algebra $V_L^{\otimes k}$ (cf. \cite{Li}), with   
 $$Y^{\hnu}(x_{\alpha}(-1)\cdot({\bf 1}\otimes \cdots \otimes {\bf 1}), z)=\sum_{m \in \frac{1}{v}\mathbb{Z}}x_{\alpha}^{\hnu} (m)z^{-m-1}.$$ 

\section{Principal Subspaces}
In this section we define the notion principal subspace of $L ^{\hnu}(k \Lambda_0)$ and twisted quasi-particles which we will use in the description of our bases.

First, denote by 
\begin{equation*}
\n=\coprod_{\alpha\in\Delta_{+}}\mathbb{C}x_\alpha,\end{equation*}
the $\nu$-stable Lie subalgebra of $\g$ (the nilradical of a Borel subalgebra), its $\hnu$-twisted affinization
\begin{equation*}
\hat{\n}[\hnu]=\coprod_{m\in\frac{1}{v}\mathbb{Z}}\n_{(vm)}\otimes t^{m}\oplus\mathbb{C}c 
\end{equation*}
 and the subalgebra of $\hat{\n}[\hnu]$ 
\begin{equation*}\begin{aligned}
\overline{\n}[\hnu]&=\coprod_{m\in\frac{1}{v}\mathbb{Z}}\n_{(vm)}\otimes t^m.
\end{aligned}\end{equation*}

Following \cite{FS1}--\cite{FS2} (see also \cite{CalLM4}, \cite{CalMPe}, and \cite{PS1}--\cite{PS2}), we define the
principal subspace $W_{L_k}^T$ of $L ^{\hnu}(k \Lambda_0)$ as: 
\begin{equation*}
W_{L_k}^T=U\left(\overline{\n}[\hnu]\right)\cdot v_{L}.
\end{equation*}

\subsection{Properties of $W_{L_k}^T$}
We now recall some important properties of the operators $\xa{}{m}$ on $W_{L_k}^T$. We recall from \cite{PS1}--\cite{PS2} that
\begin{equation*}\begin{aligned}
x_{\nu\a}^{\hnu}(m)&=\xa{}{m} \text{ for } m\in\ZZ\\
x_{\nu\a}^{\hnu}(m)&=-\xa{}{m} \text{ for } m\in\frac{1}{2}+\ZZ.\end{aligned}\end{equation*}
when $v=2$ and that
\begin{equation*}\begin{aligned}
x_{\a_3}^{\hnu}(m)&=\xa{1}{m} \text{ for } m\in\ZZ\\
x_{\a_3}^{\hnu}(m)&=\eta\xa{1}{m} \text{ for } m\in\frac{1}{3}+\ZZ\\
x_{\a_3}^{\hnu}(m)&=\eta^2\xa{1}{m} \text{ for } m\in\frac{2}{3}+\ZZ
\end{aligned}\end{equation*}
and
\begin{equation*}\begin{aligned}
x_{\a_4}^{\hnu}(m)&=\xa{1}{m} \text{ for } m\in\ZZ\\
x_{\a_4}^{\hnu}(m)&=\eta^2\xa{1}{m} \text{ for } m\in\frac{1}{3}+\ZZ\\
x_{\a_4}^{\hnu}(m)&=\eta\xa{1}{m} \text{ for } m\in\frac{2}{3}+\ZZ
\end{aligned}\end{equation*}
when $v=3$. In particular, we need only choose one representative from the orbit of each simple root $\alpha_i$ when working with the operators $x_{\a_i}^{\hnu}(m)$.

For every simple root $\alpha_i$ consider the one-dimensional subalgebra of $\mathfrak{g}$
\begin{equation*}
\n_{\alpha_i}= \mathbb{C}x_{\alpha_i},
\end{equation*} 
and their $\hnu$-twisted affinizations
\begin{equation*}
\overline{\n}_{\alpha_i}[\hnu]=\coprod_{m\in\frac{1}{v}\mathbb{Z}}{\n_{\alpha_i}}_{(vm)}\otimes t^m.
\end{equation*} 
In the case of $A_{2l-1}^{(2)}$ we define a special subspace of $\overline{\n}[\hnu]$ 
\begin{equation*}
U=U(\overline{\n}_{\alpha_l}[\hnu])U(\overline{\n}_{\alpha_{l-1}}[\hnu])  \cdots U(\overline{\n}_{\alpha_1}[\hnu]). 
\end{equation*} 
 Similary, for $D_{l}^{(2)}$ we define
\begin{equation*}
U=U(\overline{\n}_{\alpha_{l-1}}[\hnu])U(\overline{\n}_{\alpha_{l-2}}[\hnu])  \cdots U(\overline{\n}_{\alpha_1}[\hnu]), 
\end{equation*} 
in the case of $E_{6}^{(2)}$ we define
\begin{equation*}
U=U(\overline{\n}_{\alpha_{4}}[\hnu])U(\overline{\n}_{\alpha_{3}}[\hnu])U(\overline{\n}_{\alpha_{2}}[\hnu])U(\overline{\n}_{\alpha_1}[\hnu]),
\end{equation*} 
and for  $D_{4}^{(3)}$ we define
\begin{equation*}
U=U(\overline{\n}_{\alpha_{2}}[\hnu])U(\overline{\n}_{\alpha_1}[\hnu]) .
\end{equation*} 
The next lemma can be proved by using properties stated above and by aruing as in Lemma 3.1 of \cite{G}.
\begin{lem} 
In all of the above cases, we have that
\begin{equation*}
W_{L_k}^T=U \cdot v_{L}.
\end{equation*}
\end{lem}

\subsection{Twisted quasi-particles}
For each simple root $\alpha_i$, $r \in \mathbb{N}$, and $m \in \frac{1}{v}\mathbb{Z}$ define the twisted quasi-particle of {\em color} $i$, {\em charge} $r$ and {\em energy} $-m$ by
\begin{equation*} 
x^{\hnu}_{r\alpha_i}(m)=\text{Res}_z \{z^{m+r-1} x^{\hnu}_{r\alpha_i}(z)\},
\end{equation*}
where 
\begin{equation*}
x^{\hnu}_{r\alpha_i}(z)=\sum_{m\in \frac{1}{v}\mathbb{Z}}x^{\hnu}_{r\alpha_i}(m)z^{-m-r}
\end{equation*}
is the twisted vertex operator  
\begin{equation*} 
x^{\hnu}_{r\alpha_i}(z)=Y^{\hnu}(x_{\alpha_i}(-1)^r{\bf 1}, z).
\end{equation*}

As in the untwisted case (see \cite{G}, \cite{Bu1}--\cite{Bu3}), we build twisted quasi-particle monomials from twisted quasi-particles. We say that the monomial
\begin{equation*}b= b^{\hnu}(\alpha_{l})\cdots b^{\hnu}(\alpha_{1})=\end{equation*}
\begin{equation*}=x^{\hnu}_{n_{r_{l}^{(1)},l}\alpha_{l}}(m_{r_{l}^{(1)},l}) \cdots  x^{\hnu}_{n_{1,l}\alpha_{l}}(m_{1,l})\cdots 
x^{\hnu}_{n_{r_{1}^{(1)},1}\alpha_{1}}(m_{r_{1}^{(1)},1}) \cdots  x^{\hnu}_{n_{1,1}\alpha_{1}}(m_{1,1}),\end{equation*} 
is of {\em charge-type} 
\begin{equation*}\mathcal{R}'=\left(n_{r_{l}^{(1)},l}, \ldots ,n_{1,l};\ldots ;n_{r_{1}^{(1)},1}, \ldots ,n_{1,1}\right),\end{equation*} where $0 \leq n_{r_{i}^{(1)},i}\leq \ldots \leq  n_{1,i}$, {\em dual-charge-type}
\begin{equation*} \mathcal{R}=\left(r^{(1)}_{l},\ldots , r^{(s_{l})}_{l};\ldots ;r^{(1)}_{1},\ldots , r^{(s_{1})}_{1} \right),\end{equation*}
where $r^{(1)}_{i}\geq r^{(2)}_{i}\geq \ldots \geq  r^{(s_{i})}_{i}\geq 0$, and {\em color-type}
\begin{equation*} \left(r_{l},\ldots, r_{1}\right),\end{equation*}
where 
\begin{equation*} 
r_i=\sum_{p=1}^{r_{i}^{(1)}}n_{p,i}=\sum^{s_{i}}_{t=1}r^{(t)}_{i} \ \ \text{and} \ \ s_{i}\in \mathbb{N},
\end{equation*}
if for every color $i$, $1 \leq i \leq l$,
$\left(n_{r_{i}^{(1)},i}, \ldots ,n_{1,i}\right)$ and $\left(r^{(1)}_{i}, r^{(2)}_{i}, \ldots , 
r^{(s)}_{i}\right)$ are mutually conjugate partitions of $r_i$ (cf. \cite{Bu1}--\cite{Bu3}, \cite{G}). 

For two monomials $b$ and $\overline{b}$ with charge-types $\mathcal{R}'$ and $\overline{\mathcal{R}'}=\left(\overline{n}_{\overline{r}_{l}^{(1)},l}, \ldots ,\overline{n}_{1,l};\ldots;\right.$$\left.\ldots;\overline{n}_{\overline{r}_{1}^{(1)},1}, \ldots ,\bar{n}_{1,1}\right)$ and with energies $\left(m_{r_{l}^{(1)},l},\ldots , m_{1,1}\right)$ and $\left(\overline{m}_{\overline{r}_{l}^{(1)},l}, \ldots
,\overline{m}_{1,1}\right)$ (which we write so that energies of twisted quasi-particles of the same color and the same charge form an increasing sequence of integers from right to the left), respectively, we write $b < \bar{b}$ if one of the following conditions holds:
\begin{itemize}
\item[1.]\label{lin1} $\mathcal{R}'<\overline{\mathcal{R}'}$
\item[2.]\label{lin2} $\mathcal{R}'=\overline{\mathcal{R}'}$ and $\left(m_{r_{l}^{(1)},l},\ldots , m_{1,1}\right)
<
\left(\overline{m}_{\overline{r}_{l}^{(1)},l}, \ldots
,\overline{m}_{1,1}\right),$
\end{itemize}
where we write $\mathcal{R}'<\overline{\mathcal{R}'}$ if there exists $u \in \mathbb{N}$ such that $n_{1,i}=\overline{n}_{1,i}, n_{2,i}=\overline{n}_{2,i},\ldots , n_{u-1,i}=\overline{n}_{u-1,i},$ and either
$u=\overline{r}_{i}^{(1)}+1$ or $n_{u,i}<\overline{n}_{u,i}$, starting from color $i=1$. In the case that $\mathcal{R}' = \overline{\mathcal{R}'}$,  we apply this definition to the energies to similarly define $\left(m_{r_{l}^{(1)},l},\ldots , m_{1,1}\right)
<
\left(\overline{m}_{\overline{r}_{l}^{(1)},l}, \ldots
,\overline{m}_{1,1}\right).$

\section{Combinatorial Bases}
In this section we prove relations among twisted quasi-particles which we will use in the construction of our combinatorial bases for $W_{L_k}^T$. 

\subsection{Relations among twisted quasi-particles}
For every color $i$ we have the following relations:
\begin{equation}\label{rel1}
x^{\hnu}_{(k+1)\alpha_i}(z)=0 
\end{equation}
and  
\begin{equation}\label{rel2}
x^{\hnu}_{r\alpha_i}(z)v_{L} \in W^T_{L_k}\left[\left[ z \right]\right]
\end{equation}   
when $\hnu \alpha_i=\alpha_i$, 
\begin{equation}\label{rel3}
x^{\hnu}_{r\alpha_i}(z)v_{L} \in z^{-\frac{1}{2}} W^T_{L_k}\left[\left[ z^{\frac{1}{2}}\right]\right]\end{equation} 
when $\hnu \alpha_i\neq \alpha_i$ and $v=2$, and 
\begin{equation}\label{rel4}
x^{\hnu}_{r\alpha_i}(z)v_{L} \in z^{-\frac{2}{3}} W^T_{L_k}\left[\left[ z^{\frac{1}{3}}\right]\right]\end{equation}
when $\hnu \alpha_i\neq \alpha_i$ and $v=3$, which all follow immediately from the fact that
\begin{equation*}
x^{\hnu}_{r\alpha_i}(m)v_{L}  = 0 
\end{equation*}
whenever $m\ge 0$.
We will use also the following relations among quasi-particles of the same color:
\begin{lem}\label{rel5} Let $1\leq n\leq n'$ be fixed.
\begin{itemize}
\item[a)]\label{a} If $\hnu \alpha_i=\alpha_i$ and $M,j\in \mathbb{Z}$ are fixed, the $2n$ monomials from the set
\begin{multline*} A=\{x^{\hnu}_{n\alpha_i}(j)x^{\hnu}_{n'\alpha_i}(M-j),x^{\hnu}_{n\alpha_i}(j-1)x^{\hnu}_{n'\alpha_i}(M-j+1),
\ldots \\
 \ldots , x^{\hnu}_{n\alpha_i}(j-2n+1)x^{\hnu}_{n'\alpha_i}(M-j+2n-1)\}\end{multline*}
can be expressed as a linear combination of monomials from the set \begin{equation*} \left\{x^{\hnu}_{n\alpha_i}(m)x^{\hnu}_{n'\alpha_i}(m'):m+m'=M\right\} \setminus A\end{equation*} and monomials which have as a factor the quasi-particle $x^{\hnu}_{(n'+1)\alpha_i}(j')$, $j' \in \mathbb{Z}$.
\item[b)] If $\hnu \alpha_i\neq\alpha_i$ and $M,j\in \frac{1}{v}\mathbb{Z}$ are fixed, the $2n$ monomials from the set
\begin{multline*} B=\{x^{\hnu}_{n\alpha_i}(j)x^{\hnu}_{n'\alpha_i}(M-j),x^{\hnu}_{n\alpha_i}(j-\frac{1}{v})x^{\hnu}_{n'\alpha_i}(M-j+\frac{1}{v}),
\ldots \\
 \ldots , x^{\hnu}_{n\alpha_i}(j-\frac{2n-1}{v})x^{\hnu}_{n'\alpha_i}(M-j+\frac{2n-1}{v})\}\end{multline*}
can be expressed as a linear combination of monomials from the set \begin{equation*} \left\{x^{\hnu}_{n\alpha_i}(m)x^{\hnu}_{n'\alpha_i}(m'):m+m'=M\right\} \setminus B\end{equation*} and monomials which have as a factor the quasi-particle $x^{\hnu}_{(n'+1)\alpha_i}(j')$, $j' \in \frac{1}{v}\mathbb{Z}$.
\end{itemize}
\end{lem}
\begin{proof}
The proof of part a) is identical to the proof of Lemma 4.4 in \cite{JP}. Part b) can be proven analogously, but now for fixed $M\in \frac{1}{v}\mathbb{Z}$ and $m=\frac{j}{v}, \frac{j-1}{v}, \ldots, \frac{j-2n+1}{v}$ we have  
\begin{eqnarray}\nonumber
&\frac{1}{N!}\left(x^{\hnu}_{n\alpha_i}(z)\right)^{(N)}x^{\hnu}_{n'\alpha_i}(z)
=\sum_{M \in  \frac{1}{v}\mathbb{Z}}\left( \sum_{\substack{m,m'\in
 \frac{1}{v}\mathbb{Z}\\ m+m'=M}}\binom{-m-n}{N}x^{\hnu}_{n\alpha_i}(m)x^{\hnu}_{n'\alpha_i}(m') \right) z^{-M-n-n'-N}\\
\nonumber 
&=A^{\hnu}_1(z)x^{\hnu}_{(n'+1)\alpha_i}(z) + A^{\hnu}_2(z)\frac{d}{dz}x^{\hnu}_{(n'+1)\alpha_i}(z),
\end{eqnarray}
where $N=0, 1, \ldots, 2n-1$, $A^{\hnu}_1(z)$ and $A^{\hnu}_2(z)$ are monomals in twisted quasi-particles. From this, we obtain a system of $2n$ equations in the $2n$ uknowns
\begin{eqnarray}\nonumber &x^{\hnu}_{n\alpha}(-p-n)x^{\hnu}_{n'\alpha}(M+p+n), x^{\hnu}_{n\alpha}(-p-n- \frac{1}{v})x^{\hnu}_{n'\alpha}(M+p+n+ \frac{1}{v}),\dots \\
\nonumber
&\dots ,x^{\hnu}_{n\alpha}(-p-n-\frac{2n-1}{v})x^{\hnu}_{n'\alpha}(M+p+n+\frac{2n-1}{v}),
\end{eqnarray}
where $p=-\frac{j}{v}-n$ with  coefficient matrix
\[\begin{bmatrix}  \begin{array}{ccccc}
\binom{p}{0} & \binom{p+\frac{1}{v}}{0}&
\dots&\binom{p+\frac{2n-2}{v}}{0}& \binom{p+\frac{2n-1}{v}}{0} \\
\\
\binom{p}{1} & \binom{p+\frac{1}{v}}{1}& 
\dots& \binom{p+\frac{2n-2}{v}}{1}& \binom{p+\frac{2n-1}{v}}{1} \\
\vdots &   
\vdots &\vdots & \vdots & \vdots  \\
\binom{p}{2n-1} & \binom{p+\frac{1}{v}}{2n-1}& 
\dots&
\binom{p+\frac{2n-2}{v}}{2n-1}& \binom{p+\frac{2n-1}{v}}{2n-1}
\end{array}\end{bmatrix}.\]
By using the Chu-Vandermonde identity
\[ \sum_{q=0}^N \binom{a}{q}\binom{b}{N-q}=\binom{a+b}{N},\]
we can write the coefficient matrix as a product of two invertible matrices
\[ \left[   \begin{array}{ccccc}
\binom{p}{0} & 0&
\dots&0& 0 \\
\\
\binom{p}{1} & \binom{p}{0}&
\dots& 0& 0 \\
\vdots &  
\vdots &\vdots & \vdots & \vdots  \\
\binom{p}{2n-1} & \binom{p}{2n-2}&  
\dots&
\binom{p}{ 1}& \binom{p}{0}
\end{array} \right]\cdot \left[   \begin{array}{ccccc}
\binom{0}{0} & \binom{\frac{1}{v}}{0}&
\dots&\binom{\frac{2n-2}{v}}{0}& \binom{\frac{2n-1}{v}}{0} \\
\\
\binom{0}{1} & \binom{\frac{1}{v}}{1}& 
\dots& \binom{\frac{2n-2}{v}}{1}& \binom{\frac{2n-1}{v}}{1} \\
\vdots &   
\vdots &\vdots & \vdots & \vdots  \\
\binom{0}{2n-1} & \binom{\frac{1}{v}}{2n-1}&...&
\binom{\frac{2n-2}{v}}{2n-1}& \binom{\frac{2n-1}{v}}{2n-1}
\end{array} \right] .\]
To see that the second matrix in the product is invertible, we first, starting from row $2n-1$, multiply row $q$ of the determinant of the matrix by $\frac{(q-1)v-1}{qv}$, where $2 \leq q \leq 2n-1$, and then add to row $q+1$, which will give us 
\[ \begin{vmatrix}   \begin{array}{cccccc}
1 & 1& 1&
 
\dots&1&1 \\
\\
0 & \frac{1}{v}& \frac{2}{v} &
\dots&  \frac{2n-2}{v} &  \frac{2n-1}{v}  \\
0 & 0& \frac{1}{v^2} &
\dots&  \frac{1}{v^2}\binom{2n-2}{2} &  \frac{1}{v^2}\binom{2n-1}{2}\\
\vdots &  \vdots & 
\vdots &\vdots & \vdots & \vdots  \\
0 & 0& \frac{1}{(2n-1)v}\binom{\frac{2}{v}}{2n-2}&
\dots&
\frac{2n-3}{(2n-1)v}\binom{\frac{2n-2}{v}}{2n-2}& \frac{2n-2}{(2n-1)v}\binom{\frac{2n-1}{v}}{2n-2}
\end{array}\end{vmatrix} .\]
We proceed with this process of reducing the determinant of our matrix to the determinant of an upper triangular matrix and assume that after $r-1$ steps we have
\[ \begin{vmatrix}  \scalemath{0.9}{  \begin{array}{cccccccc}
1 & 1&  
\dots&1&1&1 &
 
\dots&1 \\
\\
0 & \frac{1}{v}&  
\dots& \frac{r-1}{v}&   \frac{r}{v}& \frac{r+1}{v} &
\dots&   \frac{2n-1}{v}  \\
 \vdots &   
\vdots &\vdots & \vdots & \vdots & \vdots & \vdots & \vdots\\
0 & 0&  
\dots& 0&  \frac{1}{v^r}& \frac{1+r}{v^r}&\dots &  \frac{1+r}{v^r}\binom{2n-1}{r}  \\
0 & 0& 
\dots& 0&  \frac{1}{r(r-1)v^{r-1}}\binom{\frac{r}{v}}{2}& \frac{1}{(r+1)v^{r-1}}\binom{\frac{r+1}{v}}{2}&\dots &  \frac{(2n-1)\cdots(2n-r)}{v^{r-1}(r+1)\cdots 3}\binom{\frac{r+1}{v}}{2}  \\
\vdots &   
\vdots &\vdots & \vdots & \vdots & \vdots & \vdots& \vdots \\
0 & 0&  
\dots&0& 
\frac{1}{(2n-1)\cdots (2n-r-1)v^{r-1}}\binom{\frac{r}{v}}{2n-1-r-1}& \frac{r(r-1)\cdots (2n-r)}{(2n-1)\cdots (2n-r-1)v^{r-1}}\binom{\frac{r+1}{v}}{2n-1-r-1}&\dots &  \frac{(2n-1)\cdots(2n-r)}{v^{r-1}(2n-1)\cdots (2n-1-r)}\binom{\frac{2n-1}{v}}{2n-1-r-1}
\end{array}}\end{vmatrix} .\]  
Now, starting from row $2n-1$, we multiply row $q$ by $\frac{(q-r)v-r}{qv}$, where $r+1 \leq q \leq 2n-1$, and add to row $q+1$ and obtain
\[ \begin{vmatrix}   \begin{array}{cccccccc}
1 & 1& 
\dots&1&1&1 & \dots&1 \\
\\
0 & \frac{1}{v}&  
\dots& \frac{r-1}{v}&   \frac{r}{v}& \frac{r+1}{v} &
\dots&   \frac{2n-1}{v}  \\
\vdots &    
\vdots &\vdots & \vdots & \vdots & \vdots & \vdots & \vdots\\
0 & 0&  
\dots& 0&  \frac{1}{v^r}& \frac{r+1}{v^r}&\dots &  \frac{r+1}{v^r}\binom{2n-1}{r}  \\
0 & 0&  
\dots& 0&  0 & \frac{1}{v^{r+1}}&\dots &  \frac{1}{v^{r+1}}\binom{2n-1}{r+1}  \\
\vdots &   
\vdots &\vdots & \vdots & \vdots & \vdots & \vdots& \vdots \\
0 & 0&  
\dots&0& 
0& \frac{1}{(2n-1)\cdots (2n-r-2)v^{r}}\binom{\frac{r+1}{v}}{2n-1-r-2}&\dots &  \frac{(2n-1)\cdots(2n-r-1)}{v^{r}(2n-1)\cdots (2n-1-r)}\binom{\frac{2n-1}{v}}{2n-1-r-2}
\end{array}\end{vmatrix} .\]
After $2n-1$ steps we get the determinant   
\[ \begin{vmatrix}   \begin{array}{ccccc}
1 & 1& 1& 
\dots&1 \\
\\
0 & \frac{1}{v}& \frac{2}{v} &
\dots &   \frac{2n-1}{v}  \\
0 & 0& \frac{1}{v^2} &
\dots &   \frac{1}{v^2} \binom{2n-1}{2} \\
\vdots &    
\vdots &\vdots & \vdots & \vdots   \\
0 & 0&  0& 
\dots&   \frac{1}{v^{ 2n-1 }}
\end{array}\end{vmatrix} =v^{-n(2n-1)},\]
from which our claim follows.\end{proof}
\begin{rem}
We note here that a more general form of this matrix appeared in \cite{PSW}, where is was called a ``generalized Pascal matrix". In \cite{PSW}, this matrix was shown to be invertible but its determinant was not explicitly computed.
\end{rem}
From Lemma \ref{rel5} it follows that if $\hnu \alpha_i=\alpha_i$,  then the $2n$ monomial vectors
$$x^{\hnu}_{n\alpha_i}(m)x^{\hnu}_{n'\alpha_i}(m')v_L,x^{\hnu}_{n\alpha_i}(m-1)x^{\hnu}_{n'\alpha_i}(m'+1)v_L,
\ldots, x^{\hnu}_{n\alpha_i}(m-2n+1)x^{\hnu}_{n'\alpha_i}(m'+2n-1)v_L$$
such that $n < n'$ can be expressed as a (finite) linear combination of the monomial vectors
$$x^{\hnu}_{n\alpha_i}(j)x^{\hnu}_{n'\alpha_i}(j')v_L \ \ \text{such that} \ \ j \leq m-2n, \ \  j' \geq m'+2n$$
and monomial vectors with a factor quasi-particle $x^{\hnu}_{(n+1)\alpha_i}(j_1)$, $j_1 \in \mathbb{Z}$. 
If $n=n'$, then the monomial vectors 
$$x^{\hnu}_{n\alpha_i}(m)x^{\hnu}_{n'\alpha_i}(m')v_L \ \ \text{such that} \ \   m'-2n\leq m\leq m'$$
can be expressed as a linear combination of monomial vectors
$$x^{\hnu}_{n\alpha_i}(j)x^{\hnu}_{n'\alpha_i}(j')v_L, \ \ \text{such that} \ \   j\leq j'-2n$$
and monomial vectors with a factor quasi-particle $x^{\hnu}_{(n+1)\alpha_i}(j_1)$, $j_1 \in \mathbb{Z}$. 

If $\hnu \alpha_i\neq \alpha_i$, then the $2n$ monomial vectors
$$x^{\hnu}_{n\alpha_i}(m)x^{\hnu}_{n'\alpha_i}(m')v_L,x^{\hnu}_{n\alpha_i}(m-1)x^{\hnu}_{n'\alpha_i}(m'+1)v_L,
\ldots, x^{\hnu}_{n\alpha_i}(m-2n+1)x^{\hnu}_{n'\alpha_i}(m'+2n-1)v_L$$
with $n < n'$ can be expressed as a (finite) linear combination of the monomial vectors
$$x^{\hnu}_{n\alpha_i}(j)x^{\hnu}_{n'\alpha_i}(j')v_L, \ \ \text{such that} \ \ j \leq m-\frac{2n}{v}, \ \  j' \geq m'+\frac{2n}{v}$$
and monomial vectors with a factor quasi-particle $x^{\hnu}_{(n+1)\alpha_i}(j_1)$, $j_1 \in \frac{1}{v}\mathbb{Z}$. 
If $n=n'$, then the monomial vectors 
$$x^{\hnu}_{n\alpha_i}(m)x^{\hnu}_{n'\alpha_i}(m')v_L, \ \ \text{such that} \ \   m'-\frac{2n}{v}\leq m\leq m'$$
can be expressed as a linear combination of the monomial vectors
$$x^{\hnu}_{n\alpha_i}(j)x^{\hnu}_{n'\alpha_i}(j')v_L, \ \ \text{such that} \ \   j\leq j'-\frac{2n}{v}$$
and monomial vectors with a factor quasi-particle $x^{\hnu}_{(n+1)\alpha_i}(j_1)$, $j_1 \in \frac{1}{v}\mathbb{Z}$.

The following lemma describes relations among quasi-particles with different colors:
\begin{lem}
Let $1 \leq n, n' \leq k$ be fixed. Then:
\begin{itemize}\label{rel6}
\item[a)] if $v=2$ and $\langle \alpha, \beta \rangle=-1=\langle \nu \alpha, \beta \rangle$, we have:
\begin{equation}\label{rel6}
(z_1-z_2)^{\text{min}\{n,n'\}}x^{\hnu}_{n\alpha}(z_1)x^{\hnu}_{n'\beta}(z_2)=(z_1-z_2)^{\text{min}\{n,n'\}}x^{\hnu}_{n'\beta}(z_2)x^{\hnu}_{n\alpha}(z_1),
\end{equation}
\item[b)] if $v=2$ and $\langle \alpha, \beta \rangle=-1$, $\langle \nu \alpha, \beta \rangle=0$, we have:\begin{equation}\label{rel7}(z_1^{\frac{1}{2}}-z_2^{\frac{1}{2}})^{\text{min}\{n,n'\}}x^{\hnu}_{n\alpha}(z_1)x^{\hnu}_{n'\beta}(z_2)=(z_1^{\frac{1}{2}}-z_2^{\frac{1}{2}})^{\text{min}\{n,n'\}}x^{\hnu}_{n'\beta}(z_2)x^{\hnu}_{n\alpha}(z_1), \end{equation} 
\item[c)] if $v=3$ for $\alpha=\alpha_1$, $\beta =\alpha_2$, we have:
\begin{equation}\label{rel8}
(z_1 -z_2 )^{\text{min}\{n,n'\}}x^{\hnu}_{n\alpha}(z_1)x^{\hnu}_{n'\beta}(z_2)=(z_1 -z_2 )^{\text{min}\{n,n'\}}x^{\hnu}_{n'\beta}(z_2)x^{\hnu}_{n\alpha}(z_1).
\end{equation}
\end{itemize}
\end{lem}
\begin{proof}
The proof follows from the commutator formula for twisted vertex operators (\ref{comform}) and from properties of the $\delta$-function. 
\end{proof}

By using relations (\ref{rel1})--(\ref{rel4}) among twisted quasi-particles, Lemma \ref{rel5} and relations (\ref{rel6})--(\ref{rel8}), we define the following sets:
\begin{itemize}
\item  for $A_{2l-1}^{(2)}$:
\begin{equation*} 
B= \bigcup_{\substack{n_{r_{1}^{(1)},1}\leq \ldots \leq n_{1,1}\leq 
k\\\substack{ \ldots \\ n_{r_{l}^{(1)},l}\leq \ldots \leq n_{1,l}\leq k}}}\left(\text{or, equivalently,} \ \ \ 
\bigcup_{\substack{r_{1}^{(1)}\geq \cdots\geq r_{1}^{(k)}\geq 0\\ \cdots \\r_{l}^{(1)}\geq \cdots\geq r_{l}^{(k)}\geq 
0}}\right)
\end{equation*}
\begin{equation*}\left\{b\right.= b(\alpha_{l})\cdots b(\alpha_{1})
=x^{\hnu}_{n_{r_{n}^{(1)},l}\alpha_{l}}(m_{r_{l}^{(1)},l})\cdots x^{\hnu}_{n_{1,l}\alpha_{l}}(m_{1,l})\cdots x^{\hnu}_{n_{r_{1}^{(1)},1}\alpha_{1}}(m_{r_{1}^{(1)},1})\cdots  x^{\hnu}_{n_{1,1}\alpha_{1}}(m_{1,1}):\end{equation*}
\begin{align}\nonumber
\left|
\begin{array}{l}
m_{p,i} \in \frac{1}{2}\mathbb{Z}, \ \ \  \ 1\leq  p\leq r_{i}^{(1)}, \ 1 \leq i \leq l-1;\\
m_{p,i}\leq  -\frac{n_{p,i}}{2} + \frac{1}{2} \sum_{q=1}^{r_{i-1}^{(1)}}\text{min}\left\{n_{q,i-1},n_{p,i}\right\}- \sum_{p>p'>0} 
 \ \text{min}\{n_{p,i}, n_{p',i}\}, \ 1\leq  p\leq r_{i}^{(1)}, \ 1 \leq i \leq n-1;\\
m_{p+1,i} \leq   m_{p,i}-n_{p,i} \  \text{if} \ n_{p+1,i}=n_{p,i}, \ 1\leq  p\leq r_{i}^{(1)}-1, \ 1 \leq i \leq n-1;\\
m_{p,l} \in \mathbb{Z}, \ \ \  \ 1\leq  p\leq r_{l}^{(1)};\\
m_{p,l}\leq  -n_{p,l} + \sum_{q=1}^{r_{l-1}^{(1)}}\text{min}\left\{ n_{q,l-1},n_{p,l}\right\} - \sum_{p>p'>0} 2 \ \text{min}\{n_{p,l}, n_{p',l}\}, \  1\leq  p\leq r_{l}^{(1)};\\
m_{p+1,l}\leq  m_{p,l}-2n_{p,l} \  \text{if} \ n_{p,l}=n_{p+1,l}, \  1\leq  p\leq r_{l}^{(1)}-1  
\end{array}\right\},
\end{align}
\item for $D_l^{(2)}$:
\begin{equation*} 
B= \bigcup_{\substack{n_{r_{1}^{(1)},1}\leq \ldots \leq n_{1,1}\leq 
k\\\substack{ \ldots }\\n_{r_{l-1}^{(1)},l-1}\leq \ldots \leq n_{1,l-1}\leq k}}\left(\text{or, equivalently,} \ \ \ 
\bigcup_{\substack{r_{1}^{(1)}\geq \cdots\geq r_{1}^{(k)}\geq 0\\\substack{\cdots  }\\r_{l-1}^{(1)}\geq \cdots\geq r_{l-1}^{(k)}\geq 
0}}\right)
\end{equation*}
\begin{equation*}\left\{b\right.= b(\alpha_{l-1})\cdots b(\alpha_{1})
=x^{\hnu}_{n_{r_{l-1}^{(1)},l-1}\alpha_{l-1}}(m_{r_{l-1}^{(1)},l-1})\cdots x^{\hnu}_{n_{1,l-1}\alpha_{l-1}}(m_{1,l-1})\cdots x^{\hnu}_{n_{r_{1}^{(1)},1}\alpha_{1}}(m_{r_{1}^{(1)},1})\cdots  x^{\hnu}_{n_{1,1}\alpha_{1}}(m_{1,1}):\end{equation*}
\begin{align}\nonumber
\left|
\begin{array}{l}
m_{p,i} \in\mathbb{Z}, \ \ \  \ 1\leq  p\leq r_{i}^{(1)}, \ 1 \leq i \leq l-2;\\
m_{p,i}\leq  - n_{p,i} +  \sum_{q=1}^{r_{i-1}^{(1)}}\text{min}\left\{n_{q,i-1},n_{p,i}\right\}- \sum_{p>p'>0} 2
 \ \text{min}\{n_{p,i}, n_{p',i}\}, \ 1\leq  p\leq r_{i}^{(1)}, \ 1 \leq i \leq l-2;\\
m_{p+1,i} \leq   m_{p,i}-2n_{p,i} \  \text{if} \ n_{p+1,i}=n_{p,i}, \ 1\leq  p\leq r_{i}^{(1)}-1, \ 1 \leq i \leq l-2;\\
m_{p,l-1} \in  \frac{1}{2}\mathbb{Z}, \ \ \  \ 1\leq  p\leq r_{l-1}^{(1)};\\
m_{p,l-1}\leq  -\frac{n_{p,l-1}}{2} + \sum_{q=1}^{r_{l-2}^{(1)}}\text{min}\left\{ n_{q,l-2},n_{p,l-1}\right\} - \sum_{p>p'>0}  \ \text{min}\{n_{p,l-1}, n_{p',l-1}\}, \  1\leq  p\leq r_{l-1}^{(1)};\\
m_{p+1,l-1}\leq  m_{p,l-1}-n_{p,l-1} \  \text{if} \ n_{p,l-1}=n_{p+1,l-1}, \  1\leq  p\leq r_{l-1}^{(1)}-1  
\end{array}\right\},
\end{align}
\item for $E_6^{(2)}$:
\begin{equation*} 
B= \bigcup_{\substack{n_{r_{1}^{(1)},1}\leq \ldots \leq n_{1,1}\leq 
k\\\substack{ \ldots \\ n_{r_{4}^{(1)},4}\leq \ldots \leq n_{1,4}\leq k}}}\left(\text{or, equivalently,} \ \ \ 
\bigcup_{\substack{r_{1}^{(1)}\geq \cdots\geq r_{1}^{(k)}\geq 0\\\substack{ \ldots \\ r_{4}^{(1)}\geq \cdots\geq r_{4}^{(k)}\geq 
0}}}\right)
\end{equation*}
\begin{equation*}
\left\{b\right.= b(\alpha_{4})\cdots b(\alpha_{1})
=x^{\hnu}_{n_{r_{4}^{(1)},4}\alpha_{4}}(m_{r_{4}^{(1)},4})\cdots x^{\hnu}_{n_{1,4}\alpha_{4}}(m_{1,4})\cdots x^{\hnu}_{n_{r_{1}^{(1)},1}\alpha_{1}}(m_{r_{1}^{(1)},1})\cdots  x^{\hnu}_{n_{1,1}\alpha_{1}}(m_{1,1}):
\end{equation*}
\begin{align}\nonumber
\left|
\begin{array}{l}
m_{p,i} \in \frac{1}{2}\mathbb{Z}, \ \ \  \ 1\leq  p\leq r_{i}^{(1)}, \ 1 \leq i \leq 2;\\
m_{p,i}\leq  -\frac{n_{p,i}}{2} + \frac{1}{2} \sum_{q=1}^{r_{i-1}^{(1)}}\text{min}\left\{n_{q,i-1},n_{p,i}\right\}- \sum_{p>p'>0} 
 \ \text{min}\{n_{p,i}, n_{p',i}\}, \ 1\leq  p\leq r_{i}^{(1)}, \ 1 \leq i \leq n-1;\\
m_{p+1,i} \leq   m_{p,i}-n_{p,i} \  \text{if} \ n_{p+1,i}=n_{p,i}, \ 1\leq  p\leq r_{i}^{(1)}-1, \ 1 \leq i \leq 2;\\
m_{p,i} \in \mathbb{Z}, \ \ \  \ 1\leq  p\leq r_{i}^{(1)}, \ 3 \leq i \leq 4;\\
m_{p,i}\leq  -n_{p,i} + \sum_{q=1}^{r_{i-1}^{(1)}}\text{min}\left\{ n_{q,i-1},n_{p,i}\right\} - \sum_{p>p'>0} 2 \ \text{min}\{n_{p,i}, n_{p',i}\}, \  1\leq  p\leq r_{i}^{(1)} , \ 3 \leq i \leq 4;\\
m_{p+1,i}\leq  m_{p,i}-2n_{p,i} \  \text{if} \ n_{p,i}=n_{p+1,i}, \  1\leq  p\leq r_{i}^{(1)}-1, \ 3 \leq i \leq 4  
\end{array}\right\}.
\end{align}
\item for $D_4^{(3)}$:
\begin{equation*} 
B= \bigcup_{\substack{n_{r_{1}^{(1)},1}\leq \ldots \leq n_{1,1}\leq 
k\\\substack{   n_{r_{2}^{(1)},2}\leq \ldots \leq n_{1,2}\leq k}}}\left(\text{or, equivalently,} \ \ \ 
\bigcup_{\substack{r_{1}^{(1)}\geq \cdots\geq r_{1}^{(k)}\geq 0\\\substack{ r_{2}^{(1)}\geq \cdots\geq r_{2}^{(k)}\geq 
0}}}\right)
\end{equation*}
\begin{equation*}
\left\{b\right.= b(\alpha_{2}) b(\alpha_{1})
=x^{\hnu}_{n_{r_{2}^{(1)},2}\alpha_{2}}(m_{r_{2}^{(1)},2})\cdots x^{\hnu}_{n_{1,2}\alpha_{2}}(m_{1,2})  x^{\hnu}_{n_{r_{1}^{(1)},1}\alpha_{1}}(m_{r_{1}^{(1)},1})\cdots  x^{\hnu}_{n_{1,1}\alpha_{1}}(m_{1,1}):
\end{equation*}
\begin{align}\nonumber
\left|
\begin{array}{l}
m_{p,1} \in \frac{1}{3}\mathbb{Z}, \ \ \  \ 1\leq  p\leq r_{1}^{(1)};\\
m_{p,1}\leq  - \frac{n_{p,1}}{3} - \frac{2}{3}\sum_{p>p'>0} 
 \ \text{min}\{n_{p,1}, n_{p',1}\}, \ 1\leq  p\leq r_{1}^{(1)};\\
m_{p+1,1} \leq   m_{p,1}-\frac{2}{3}n_{p,1} \  \text{if} \ n_{p+1,1}=n_{p,1}, \ 1\leq  p\leq r_{1}^{(1)}-1;\\
m_{p,2} \in \mathbb{Z}, \ \ \  \ 1\leq  p\leq r_{2}^{(1)};\\
m_{p,2}\leq  -n_{p,2} + \sum_{q=1}^{r_{1}^{(1)}}\text{min}\left\{ n_{q,1},n_{p,2}\right\} - \sum_{p>p'>0} 2 \ \text{min}\{n_{p,2}, n_{p',2}\}, \  1\leq  p\leq r_{2}^{(1)};\\
m_{p+1,2}\leq  m_{p,2}-2n_{p,2} \  \text{if} \ n_{p,2}=n_{p+1,2}, \  1\leq  p\leq r_{2}^{(1)}-1
\end{array}\right\},
\end{align}
\end{itemize}
where $r^{(1)}_0 := 0$. 

Now, using the same proof as in \cite{G}, we have:
\begin{prop}
The set 
\begin{equation}\label{span}
\mathcal{B}=\{ bv_L: b \in B\}
\end{equation}
spans the principal subspace $W_{L_k}^T$.
\end{prop}

\section{Proof of linear independence}
To prove that the set $\mathcal{B}$ is a linearly independent set, we will use certain maps defined on our principal subspace. The first of these maps is a generalization of the projection map found in \cite{G} to our twisted setting, and the remaining maps were used in \cite{PS1} and \cite{PS2}, (see also \cite{L1} and \cite{CalLM4}). The proof of the linear independence of $\mathcal{B}$ is carried out by induction on the linear order on the twisted quasi-particle monomials.
\subsection{The projection $\pi_{\mathcal{R}}$} As we mentioned earlier, for $k \geq 1$, we will realize $W_{L_k}^T$ as a subspace of the tensor product of $k$ level one modules $W_{L}^T$
\[ W_{L_k}^T \subset W_{L}^T \otimes \cdots \otimes W_{L}^T \subset {V_L^T} ^{\otimes k}.\]
For a chosen dual-charge-type
$$\mathcal{R}=\left(r^{(1)}_{l},\ldots , r^{(k)}_{l};\ldots ;r^{(1)}_{1},\ldots , r^{(k)}_{1} \right),$$
we define a projection $\pi_{\mathcal{R}}$ of $W_{L_k}^T$ on 
\[ W_{L_{(r_l^{(k)}, \ldots, r_1^{(k)})}}^T \otimes \cdots \otimes W_{L_{(r_l^{(1)}, \ldots, r_1^{(1)})}}^T,\]
where $W_{L_{(r_l^{(t)}, \ldots, r_1^{(t)})}}^T$ denotes the subspace of $W_{L}^T$ consisting of the vectors of charges $r_l^{(t)}, \ldots, r_1^{(t)}$, for $1 \leq t \leq k$ (see \ref{charge1}--\ref{charge4}). We will also consider a generalization of $\pi_{\mathcal{R}}$, which we denote by the same symbol, to the space of formal series with coefficients in ${W_{L}^T}^{\otimes k}$.

Using the level $1$ relation $x^{\hnu}_{2\alpha_i}(z)=0$, with the projection $\pi_{\mathcal{R}}$ for fixed $1 \leq p \leq r_{i}^{(1)}$,  we ``place'' $n_{p,i}$, generating functions $x^{\hnu}_{\alpha_i}(z)$ on first  $n_{p,i}$ tensor factors:
\[ x^{\hnu}_{n^{(k)}_{p,i}\alpha_{i}}(z_{p,i}){\bf 1}_T\otimes  x^{\hnu}_{n^{(k-1)}_{p,i}\alpha_{i}}(z_{p,i}){\bf 1}_T \otimes \cdots \otimes      x^{\hnu}_{n^{(2)}_{p,i}\alpha_{i}}(z_{p,i}){\bf 1}_T\otimes  x^{\hnu}_{n^{(1)}_{p,i}\alpha_{i}}(z_{p,i}){\bf 1}_T,
\] 
where
\[  0 \leq  n^{(t)}_{p,i} \leq 1, \ \  \  n_{p,i}=\sum_{t=1}^k n^{(t)}_{p,i}, \ \ \ 1 \leq t \leq k. \]
Now, it follows that the projection of generating function
\begin{equation}\label{gf1}
x^{\hnu}_{n_{r_{l}^{(1)},l}\alpha_{l}}(z_{r_{l}^{(1)},l})\cdots  x^{\hnu}_{n_{1,1}\alpha_{1}}(z_{1,1}) \ v_{L}
\end{equation}
of dual-charge-type $\mathcal{R}$ and corresponding charge-type  $\mathcal{R}'=\left(n_{r_{l}^{(1)},l}, \ldots ,n_{1,l};\ldots ;n_{r_{1}^{(1)},1}, \ldots ,n_{1,1}\right)$ is
\begin{align*}
\pi_{\mathcal{R}}& x^{\hnu}_{n_{r_{l}^{(1)},l}\alpha_{l}}(z_{r_{l}^{(1)},l})\cdots  x^{\hnu}_{n_{1,1}\alpha_{1}}(z_{1,1}) \ v_{L}&\\
\nonumber
=&\text{C} \ x^{\hnu}_{n^{(k)}_{r^{(k)}_{l},l}\alpha_{l}}(z_{r_{l}^{(k)},l})\cdots  x^{\hnu}_{n^{(k)}_{1,l}\alpha_{l}}(z_{1,l}) \cdots  x^{\hnu}_{n^{(k)}_{r^{(k)}_{1},1}\alpha_{1}}(z_{r_{1}^{(k)},1})\cdots   x^{\hnu}_{n^{(k)}_{1,1}\alpha_{1}}(z_{1,1}) \ {\bf 1}_T&\\
\nonumber
& \ \ \ \ \ \ \ \ \ \ \ \ \ \ \ \ \otimes \ldots \otimes&\\
\nonumber
\otimes &x^{\hnu}_{n_{r^{(1)}_{l},l}^{(1)}\alpha_{l}}(z_{r_{l}^{(1)},l})\cdots    x^{\hnu}_{n_{1,l}^{(1)}\alpha_{l}}(z_{1,l})\cdots x^{\hnu}_{n_{r^{(1)}_{1},1}^{(1)}\alpha_{1}}(z_{r_{1}^{(1)},1})\cdots     x^{\hnu}_{n{_{1,1}^{(1)}\alpha_{1}}}(z_{1,1}) \ {\bf 1}_T,&
\end{align*}
where $\text{C} \in \mathbb{C}^{*}$,
and 
\begin{equation*}
0 \leq  n^{(t)}_{p,i} \leq 1, \ \ \ 1 \leq t \leq k, \ \ \ n^{(1)}_{p,i}\geq n^{(2)}_{p,i}\geq \ldots \geq  n^{(k-1)}_{p,i}\geq  n^{(k)}_{p,i}, \ \ \  n_{p,i}=\sum_{t=1}^k n^{(t)}_{p,i},
\end{equation*}
for every every $p$, $1 \leq p \leq r_{i}^{(1)}$, $1 \leq i \leq l$.

Also, from above it follows that the projection of $bv_L$, where $b \in B$ is the monomial 
\begin{equation}\label{mgf}
x^{\hnu}_{n_{r_{l}^{(1)},l}\alpha_{l}}(m_{r_{l}^{(1)},l})\cdots  x^{\hnu}_{n_{1,1}\alpha_{1}}(m_{1,1})  
\end{equation}
of charge-type $\mathcal{R}'$ and dual-charge-type $\mathcal{R}$, is a coefficient of the generating function (\ref{gf1}), which we will denote with $\pi_{\mathcal{R}}bv_L$. 

\subsection{The maps $\Delta^{T}(\lambda,-z)$} 
Following \cite{CalLM4}, \cite{CalMPe}, and \cite{PS1}--\cite{PS2} for $\gamma\in\h_{(0)}$ and a character of the root lattice $\theta:L\to\CC$, define
\begin{equation*}\begin{aligned}
\ta{}:\overline{\mathfrak{n}}[\hnu]&\to\overline{\mathfrak{n}}[\hnu]\\
\xa{}{m}&\mapsto\theta_{}(\alpha)\xa{}{m+\left<\a_{(0)},\gamma\right>}.
\end{aligned}\end{equation*}

We note here that in general, $\tau_{\gamma, \theta}$ is a linear map, and for suitably chosen characters $\theta$, it becomes an automorphism of $\overline{\mathfrak{n}}[\hnu]$, which we extend to an automorphism of $\un$. In particular, we consider maps $\tau_{\gamma, \theta}$ where $\gamma = \gamma_i=(\lambda_{i})_{(0)} = \frac{1}{2}\left( \lambda_i + \nu \lambda_i \right)$, and where the corresponding characters $\theta = \theta_i$ are chosen as in \cite{PS1}--\cite{PS2} defined by
 \begin{equation*} \theta_i(\a_j)=(-1)^{\left<\lambda^{(i)},\a_j\right>},\end{equation*}
 where we choose our subscripts $i$ as follows:
 \begin{itemize}
\item for $A_{2l-1}$, let $i=1,\dots,l$
\item  for $D_l$ when $v=2$,  let $i=1,\dots,l-1$
\item for $D_4$ when $v=3$, let $i=1,2$
\item  $E_6$, let $i=1,2,3,4$.
\end{itemize}
and define $\lambda^{(i)} = v(\lambda_i)_{(0)}$ for $v=2,3$ (see \cite{PS2} and \cite{PSW} for a more general definition of this symbol).
 
We also consider the related map $\Delta_c(\lambda,x)$ from \cite{CalLM4}, \cite{CalMPe}, and \cite{PS1}--\cite{PS2}.\\
For $\lambda\in\{\lambda_i |1\leq i\leq l\}$, we define the map 
\begin{equation*}
 \Delta^{T}(\lambda,-z)=(-1)^{\nu \lambda}z^{\lambda_{(0)}}E^{+}(-\lambda,z).
 \end{equation*}
 and its constant term $\Delta_c^T(\lambda,-z)$. From \cite{PS1}--\cite{PS2}, we have
 \begin{equation}\label{2map2}
  \Delta_c^T(\lambda_i,-z)(\xbet{m} {\bf 1}_T) = \ta{i}(\xbet{m}){\bf 1}_T.
 \end{equation}
 More generally, we have linear maps
 \begin{equation*} 
\begin{aligned}
 \Delta_c^T(\lambda_i,-z):W_L^T&\to W_L^T\\
 a\cdot {\bf 1}_T&\mapsto \ta{i}(a)\cdot {\bf 1}_T,
 \end{aligned}\end{equation*}
 where $a\in\un$.
 
 Fix $s \leq k$ and consider the map 
 \begin{equation*} 
1\otimes\cdots \otimes  \Delta_c^T(\lambda_i,-z) \otimes \underbrace{{ 1} \otimes \cdots \otimes { 1}}_{s-1 \ \text{factors}}.
 \end{equation*}
 Let $b \in B$ as in (\ref{mgf}).  It follows that 
\begin{equation*}({ 1}\otimes \cdots  \otimes { 1} \otimes  \Delta_c^T(\lambda_i,-z) \otimes { 1} \otimes \cdots \otimes { 1})\pi_{\mathcal{R}}bv_{L}
\end{equation*}
is the coefficient of  
\begin{equation*}
({ 1}\otimes\cdots \otimes \Delta_c^T(\lambda_i,-z)  \otimes { 1} \otimes \cdots \otimes  { 1})\pi_{\mathcal{R}}x^{\hnu}_{n_{r_{l}^{(1)},l}\alpha_{l}}(z_{r_{l}^{(1)},l})\cdots x^{\hnu}_{s\alpha_{1}}(z_{1,1})v_{L},
\end{equation*}
where, from (\ref{2map2}), it follows that operator $\Delta_c^T(\lambda_i,-z)$ acts only on the $s$-th tensor row as: 
\begin{equation*}\otimes  x^{\hnu}_{n^{(s)}_{r^{(s)}_{l},l}\alpha_{l}}(z_{r_{l}^{(s)},l})\cdots    x^{\hnu}_{n^{(s)}_{1,l}\alpha_{l}}(z_{1,l})\cdots x^{\hnu}_{n^{(s)}_{r^{(s)}_{i},i}\alpha_{i}}(z_{r_{i}^{(s)},i})z_{r_{i}^{(s)},i}\cdots    x^{\hnu}_{n^{(s)}_{1,i}\alpha_{l}}(z_{1,i})z_{1,i}\cdots\end{equation*}
\begin{equation*}\cdots x^{\hnu}_{n^{(s)}_{r^{(s)}_{1},1}\alpha_{1}}(z_{r_{1}^{(s)},1})\cdots  x^{\hnu}_{\alpha_{1}}(z_{1,1}){\bf 1}_T\otimes,\end{equation*}
for $i$ such that $\nu \alpha_i=\alpha_i$ and as 
\begin{equation*}\otimes  x^{\hnu}_{n^{(s)}_{r^{(s)}_{l},l}\alpha_{l}}(z_{r_{l}^{(s)},l})\cdots    x^{\hnu}_{n^{(s)}_{1,l}\alpha_{l}}(z_{1,l})\cdots x^{\hnu}_{n^{(s)}_{r^{(s)}_{i},i}\alpha_{i}}(z_{r_{i}^{(s)},i})z^{\frac{1}{v}}_{r_{i}^{(s)},i}\cdots    x^{\hnu}_{n^{(s)}_{1,i}\alpha_{l}}(z_{1,i})z^{\frac{1}{v}}_{1,i}\cdots\end{equation*}
\begin{equation*}\cdots x^{\hnu}_{n^{(s)}_{r^{(s)}_{1},1}\alpha_{1}}(z_{r_{1}^{(s)},1})\cdots  x^{\hnu}_{\alpha_{1}}(z_{1,1}){\bf 1}_T\otimes,\end{equation*}
if $i$ is such that $\nu \alpha_i\neq \alpha_i$, where $ 0 \leq  n^{(s)}_{p,i} \leq 1$, for $1\leq  p \leq  r^{(s)}_{i}$.
By taking the corresponding coefficients, we have
\begin{equation*} 
({1}\otimes \cdots  \otimes {1} \otimes  \Delta_c^T(\lambda_i,-z) \otimes { 1} \otimes \cdots \otimes { 1})\pi_{\mathcal{R}}bv_{L}=\pi_{\mathcal{R}}b^+v_{L},
\end{equation*}
where 
\begin{equation*} 
b^+=b(\alpha_l)\cdots b(\alpha_{i+1}) b^+(\alpha_{i}) b(\alpha_{i-1})\cdots  b(\alpha_{1}),
\end{equation*}
with
\begin{equation*}
b^+(\alpha_{i}) =x^{\hnu}_{n_{r_{i}^{(1)},i}\alpha_{i}}(m_{r_{i}^{(1)},i}+1)\cdots  x^{\hnu}_{n_{1,i}\alpha_{i}}(m_{1,i}+1)  
\end{equation*}
if $\nu \alpha_i=\alpha_i$ and with
\begin{equation*}
b^+(\alpha_{i}) =x^{\hnu}_{n_{r_{i}^{(1)},i}\alpha_{i}}(m_{r_{i}^{(1)},i}+\frac{1}{v})\cdots  x^{\hnu}_{n_{1,i}\alpha_{i}}(m_{1,i}+\frac{1}{v})  
\end{equation*}
if $\nu \alpha_i\neq\alpha_i$.

\subsection{The maps $e_{\alpha_i}$}
Finally, we recall the maps $e_{\alpha_i}$, which satisfy
 \begin{equation*}\begin{aligned}
 e_{\a_i}&:V_L^T\to V_L^T
  \end{aligned}\end{equation*}
  and their restriction to the principal subspace $W_L^T\subset V_L^T$ where 
  \begin{equation*}\begin{aligned}
  e_{\a_i}\cdot 1&=\frac{2}{\sigma(\a_i)}\xa{i}{-1}\cdot {\bf 1}_T \text{ if } \nu \a_i = \a_i\\
    e_{\a_i}\cdot 1&=\frac{2}{\sigma(\a_i)}\xa{i}{-\frac{1}{2}}\cdot {\bf 1}_T\text{ if } \nu \a_i \neq\a_i,
 \end{aligned}\end{equation*}
 when $v=2$ and 
 \begin{equation*}\begin{aligned}
  e_{\a_1}\cdot 1&=\frac{3}{\sigma(\a_1)}\xa{1}{-\frac{1}{3}}\cdot {\bf 1}_T\\
  e_{\a_2}\cdot 1&=\frac{3}{\sigma(\a_2)}\xa{2}{-1}\cdot {\bf 1}_T
 \end{aligned}\end{equation*}
 if $v=3$,
 and commute with our operators $\xbet{n}$ via
\begin{equation*}
 e_{\a_i}\xbet{n}=C(\a_i,\b)\xbet{n-\left<\beta_{(0)},\a_i\right>}e_{\a_i}.
 \end{equation*}

Now, assume that we have a monomial
 \begin{equation*} 
b=b(\alpha_l)\cdots b(\alpha_1) x^{\hnu}_{s\alpha_1}(-\frac{s}{2}) \in B\end{equation*}
\begin{equation*} 
b=x^{\hnu}_{n_{r_{l}^{(1)},l}\alpha_{l}}(m_{r_{l}^{(1)},l})\cdots 
\cdots  x^{\hnu}_{n_{r_{1}^{(1)},1}\alpha_{1}}(m_{r_{1}^{(1)},1})  \cdots x^{\hnu}_{n_{2,1}\alpha_{1}}(m_{2,1})x^{\hnu}_{s\alpha_{1}}(-\frac{s}{2}), \end{equation*}
of dual-charge-type
\[ \mathcal{R}=\left(r^{(1)}_{l},\ldots , r^{(k)}_{l};\ldots ;r^{(1)}_{1},\ldots , r^{(s)}_{1},0, \ldots, 0 \right) \]
and a projection $\pi_{\mathcal{R}}bv_L$, which is a coefficient of
\begin{align*}
\label{projection1}
\pi_{\mathcal{R}}& x^{\hnu}_{n_{r_{l}^{(1)},l}\alpha_{l}}(z_{r_{l}^{(1)},l})\cdots  x^{\hnu}_{n_{2,1}\alpha_{1}}(z_{1,1}) \ {\bf 1}_T\otimes\cdots \otimes {\bf 1}_T \otimes x^{\hnu}_{\alpha_{1}}(-\frac{1}{2}) {\bf 1}_T\otimes \cdots \otimes x^{\hnu}_{\alpha_{1}}(-\frac{1}{2}) {\bf 1}_T&\\
\nonumber
=&\text{C} \ x^{\hnu}_{n^{(k)}_{r^{(k)}_{l},l}\alpha_{l}}(z_{r_{l}^{(k)},l})\cdots  x^{\hnu}_{n^{(k)}_{1,l}\alpha_{l}}(z_{1,l}) \cdots  x^{\hnu}_{n^{(k)}_{r^{(k)}_{1},1}\alpha_{1}}(z_{r_{1}^{(k)},1})\cdots   x^{\hnu}_{n^{(k)}_{2,1}\alpha_{1}}(z_{1,1}) \ {\bf 1}_T&\\
\nonumber
& \ \ \ \ \ \ \ \ \ \ \ \ \ \ \ \ \otimes \ldots \otimes&\\
\nonumber
&\ x^{\hnu}_{n^{(s+1)}_{r^{(s+1)}_{l},l}\alpha_{l}}(z_{r_{l}^{(s+1)},l})\cdots  x^{\hnu}_{n^{(s+1)}_{1,l}\alpha_{l}}(z_{1,l}) \cdots  x^{\hnu}_{n^{(s+1)}_{r^{(s+1)}_{1},1}\alpha_{1}}(z_{r_{1}^{(s+1)},1})\cdots   x^{\hnu}_{n^{(s+1)}_{2,1}\alpha_{1}}(z_{2,1}) \ {\bf 1}_T&\\
\nonumber
&\ x^{\hnu}_{n^{(s)}_{r^{(s)}_{l},l}\alpha_{l}}(z_{r_{l}^{(s)},l})\cdots  x^{\hnu}_{n^{(s)}_{1,l}\alpha_{l}}(z_{1,l}) \cdots  x^{\hnu}_{n^{(s)}_{r^{(s)}_{1},1}\alpha_{1}}(z_{r_{1}^{(s)},1})\cdots   x^{\hnu}_{n^{(s)}_{2,1}\alpha_{1}}(z_{2,1}) \ e_{\a_1}{\bf 1}_T &\\
\nonumber
& \ \ \ \ \ \ \ \ \ \ \ \ \ \ \ \ \otimes \ldots \otimes&\\
\nonumber
\otimes &x^{\hnu}_{n_{r^{(1)}_{l},l}^{(1)}\alpha_{l}}(z_{r_{l}^{(1)},l})\cdots    x^{\hnu}_{n_{1,l}^{(1)}\alpha_{l}}(z_{1,l})\cdots x^{\hnu}_{n_{r^{(1)}_{1},1}^{(1)}\alpha_{1}}(z_{r_{1}^{(1)},1})\cdots     x^{\hnu}_{n{_{2,1}^{(1)}\alpha_{1}}}(z_{2,1}) \ e_{\a_1}{\bf 1}_T ,&
\end{align*} 
where $\text{C} \in \mathbb{C}^{*}$. 

If we move the operator $ 1\otimes\cdots \otimes 1 \otimes \underbrace{e_{\a_1} \otimes \cdots \otimes e_{\a_1}}_{s \ \text{factors}}$ all the way to the left we will get a projection $\pi_{\mathcal{R^{-}}}b'v_L$, where
\[ \mathcal{R}^{-}=\left(r^{(1)}_{l},\ldots , r^{(k)}_{l};\ldots ;r^{(1)}_{1},\ldots , r^{(s)}_{1}-1,0, \ldots, 0 \right) \]
and 
\begin{equation*} 
b'=b(\alpha_l)\cdots b'(\alpha_{2}) b'(\alpha_{1}),
\end{equation*}
with
\begin{equation*} 
b'(\alpha_{1}) =x^{\hnu}_{n_{r_{1}^{(1)},1}\alpha_{1}}(m_{r_{1}^{(1)},1}+n^{(1)}_{r_{1}^{(1)},1}+\cdots +n^{(s)}_{r_{1}^{(1)},1})\cdots  x^{\hnu}_{n_{2,1}\alpha_{1}}(m_{2,1}+n^{(1)}_{2,1}+\cdots +n^{(s)}_{2,1})  
\end{equation*}
and
\begin{equation*} 
b'(\alpha_{2}) =x^{\hnu}_{n_{r_{2}^{(1)},2}\alpha_{2}}(m_{r_{2}^{(1)},2}-\frac{n^{(1)}_{r_{2}^{(1)},2}+\cdots +n^{(s)}_{r_{2}^{(1)},2}}{2})\cdots  x^{\hnu}_{n_{1,2}\alpha_{j}}(m_{1,2}-\frac{n^{(1)}_{1,2}+\cdots +n^{(s)}_{1,2}}{2}).  
\end{equation*}
Above we only considered the case when $v=2$ and $\nu \alpha_1 \neq \alpha_1$, but note here that the remaining cases are similar.

\subsection{A proof of linear independence}
Here we will prove our main theorem:
\begin{thm}
 The set 
\begin{equation*} 
\mathcal{B}=\{ bv_L: b \in B\}
\end{equation*}
is a basis of the principal subspace $W_{L_k}^T$.
 \end{thm}
\begin{proof}
By Proposition \ref{span}, the set of monomial vectors $\mathcal{B}$ spans $W_{L_k}^T$, and so it remains to show that $\mathcal{B}$ is linearly independent. To prove linear independence first assume that we have
\[bv_L=0\]
where
\[b=x^{\hnu}_{n_{r_{l}^{(1)},l}\alpha_{l}}(m_{r_{l}^{(1)},l})\cdots x^{\hnu}_{n_{1,l}\alpha_{l}}(m_{1,l})\cdots  x^{\hnu}_{n_{r_{1}^{(1)},1}\alpha_{1}}(m_{r_{1}^{(1)},1})  \cdots x^{\hnu}_{n_{1,1}\alpha_{1}}(m_{1,1})  \in B,
\]
of charge-type
\[\mathcal{R}'=\left(n_{r_{l}^{(1)},l}, \ldots ,n_{1,l};\ldots ;n_{r_{1}^{(1)},1}, \ldots ,n_{1,1}\right)
\]
and dual-charge-type
\[ \mathcal{R}=\left(r^{(1)}_{l},\ldots , r^{(k)}_{l};\ldots ;r^{(1)}_{1},\ldots , r^{(n_{1,1})}_{1} \right), \]
which determines the projection $\pi_{\mathcal{R}}$, so that we have
\begin{equation}\label{p0}
\pi_{\mathcal{R}}bv_L=0.
\end{equation}
We will assume that $\nu \alpha_1\neq\alpha_1$ and we will let $s= n_{1,1}$.
We apply $1\otimes\cdots \otimes  \Delta_c^T(\lambda_1,-z)^{d} \otimes \underbrace{1 \otimes \cdots \otimes 1}_{s-1 \ \text{factors}}$ to (\ref{p0}), where $\Delta_c^T(\lambda_1,-z)^{d}=\underbrace{\Delta_c^T(\lambda_1,-z) \circ \cdots \circ \Delta_c^T(\lambda_1,-z)}_{d \ \text{times}}$ and where $d \in \mathbb{N}$ is selected so that after application of this map the twisted quasi-particle of color 1 and charge $s$ has energy $-\frac{s}{v}$. From the considerations in Subsection 5.2, we have
\begin{equation}\label{p1}
\pi_{\mathcal{R}}x^{\hnu}_{n_{r_{l}^{(1)},l}\alpha_{l}}(m_{r_{l}^{(1)},l})\cdots x^{\hnu}_{n_{1,l}\alpha_{l}}(m_{1,l})\cdots  x^{\hnu}_{n_{r_{1}^{(1)},1}\alpha_{1}}(m^+_{r_{1}^{(1)},1} )  \cdots\end{equation}
\begin{equation*} \cdots x^{\hnu}_{n_{2,1}\alpha_{1}}(m^+_{2,1} )\left({\bf 1}_T\otimes\cdots \otimes  {\bf 1}_T  \otimes x^{\hnu}_{\alpha_{1}}(-\frac{1}{v}){\bf 1}_T  \otimes \cdots \otimes x^{\hnu}_{\alpha_{1}}(-\frac{1}{v}){\bf 1}_T\right) =0,
\end{equation*}
which is a projection of a monomial vector from $\mathcal{B}$. 
From (\ref{p1}) it follows that
\begin{equation*} 
(1\otimes\cdots \otimes 1 \otimes  e_{\a_1} \otimes \cdots \otimes e_{\a_1}) \pi_{\mathcal{R}^-}x^{\hnu}_{n_{r_{l}^{(1)},l}\alpha_{l}}(m_{r_{l}^{(1)},l})\cdots x^{\hnu}_{n_{1,l}\alpha_{l}}(m_{1,l})\cdots 
\end{equation*}
\begin{equation*}
\cdots x^{\hnu}_{n_{r_{2}^{(1)},2}\alpha_{2}}(m'_{r_{2}^{(1)},2} ) x^{\hnu}_{n_{1,2}\alpha_{2}}(m'_{1,2}) x^{\hnu}_{n_{r_{1}^{(1)},1}\alpha_{1}}(m'_{r_{1}^{(1)},1} ) x^{\hnu}_{n_{2,1}\alpha_{1}}(m'_{2,1})v_L =0,
\end{equation*}
where
\[ \mathcal{R}^-=\left(r^{(1)}_{l},\ldots , r^{(k)}_{l};\ldots ;r^{(1)}_{1},\ldots , r^{(s)}_{1}-1 \right). \]
By injectivity of $1\otimes\cdots \otimes 1 \otimes  e_{\a_1} \otimes \cdots \otimes e_{\a_1}$, we have 
\begin{equation}\label{p3}
\pi_{\mathcal{R}^-}x^{\hnu}_{n_{r_{l}^{(1)},l}\alpha_{l}}(m_{r_{l}^{(1)},l})\cdots x^{\hnu}_{n_{1,l}\alpha_{l}}(m_{1,l})\cdots 
\end{equation}
\begin{equation*}
\cdots x^{\hnu}_{n_{r_{2}^{(1)},2}\alpha_{2}}(m'_{r_{2}^{(1)},2} ) x^{\hnu}_{n_{1,2}\alpha_{2}}(m'_{1,2}) x^{\hnu}_{n_{r_{1}^{(1)},1}\alpha_{1}}(m'_{r_{1}^{(1)},1} ) x^{\hnu}_{n_{2,1}\alpha_{1}}(m'_{2,1})v_L =0.
\end{equation*}
If $v=2$, for every $2 \leq p \leq r_1^{(1)}$ such that $n_{p,1}=s'\leq s$ we have
\[ m'_{p,1}=m^+_{p,1}+s'\leq  - \frac{n_{p,1}}{2}  -  \sum_{p>p'>0}  \ \text{min}\{n_{p,1}, n_{p',1}\} +s';\]
\[ m'_{p+1,1}=m^+_{p,1}+s'\leq   m^+_{p,1}-n_{p,1}+s'=m'_{p,1}-n_{p,1} \  \text{if} \ n_{p+1,1}=n_{p,1};\]
\[ m'_{p,2}=m^+_{p,2}-\frac{s'}{2}\leq  - \frac{n_{p,2}}{2} + \frac{1}{2} \sum_{q=1}^{r_{1}^{(1)}}\text{min}\left\{n_{q,1},n_{p,2}\right\} - \sum_{p>p'>0}  \ \text{min}\{n_{p,1}, n_{p',1}\}-\frac{s'}{2};\]
\[ m'_{p+1,2}=m^+_{p,2}-\frac{s'}{2} \leq   m^+_{p,2}-n_{p,2}-\frac{s'}{2}=m'_{p,2}-n_{p,2} \  \text{if} \ n_{p+1,2}=n_{p,2}.\]
If $v=3$, for every $2 \leq p \leq r_1^{(1)}$ and for $n_{p,1}=s'\leq s$ we have 
\[ m'_{p,1}=m^+_{p,1}+\frac{2s'}{3}\leq  - \frac{n_{p,1}}{3}  - \frac{2}{3}\sum_{p>p'>0}  \ \text{min}\{n_{p,1}, n_{p',1}\}+\frac{2s'}{3};\]
\[ m'_{p+1,1}=m^+_{p,1}+\frac{2s'}{3} \leq   m^+_{p,1}-\frac{2}{3}n_{p,1}+\frac{2s'}{3}=m'_{p,1}-\frac{2}{3}n_{p,1} \  \text{if} \ n_{p+1,1}=n_{p,1};\]
\[ m'_{p,2}=m^+_{p,2}-s'\leq  - n_{p,2}  +  \sum_{q=1}^{r_{1}^{(1)}}\text{min}\left\{n_{q,1},n_{p,2}\right\} - \sum_{p>p'>0} 2 \ \text{min}\{n_{p,1}, n_{p',1}\}-s';\]
\[ m'_{p+1,2}=m^+_{p,2}-s' \leq   m^+_{p,2}-n_{p,2}-s'=m'_{p,2}-n_{p,2} \  \text{if} \ n_{p+1,2}=n_{p,2}.\]
This shows that in (\ref{p3}) we have the projection of a monomial vector from the set $\mathcal{B}$. In the case when $\nu \alpha_1=\alpha_1$, with the above procedure will end with a monomial vector as in (\ref{p3}), which is also from the set $\mathcal{B}$, since for every $2 \leq p \leq r_1^{(1)}$and for $n_{p,1}=s'\leq s$, we have
\[ m'_{p,1}=m^+_{p,1}+2s'\leq  - n_{p,1}  -  \sum_{p>p'>0} 2 \ \text{min}\{n_{p,1}, n_{p',1}\} +2s';\]
\[ m'_{p+1,1}=m^+_{p,1}+2s'\leq   m^+_{p,1}-2n_{p,1}+2s'=m'_{p,1}-2n_{p,1} \  \text{if} \ n_{p+1,1}=n_{p,1};\]
\[ m'_{p,2}=m^+_{p,2}-s'\leq  -  n_{p,2} +  \sum_{q=1}^{r_{1}^{(1)}}\text{min}\left\{n_{q,1},n_{p,2}\right\} - \sum_{p>p'>0} 2 \ \text{min}\{n_{p,1}, n_{p',1}\}-s';\]
\[ m'_{p+1,2}=m^+_{p,2}-s' \leq   m^+_{p,2}-2n_{p,2}-s'=m'_{p,2}-2n_{p,2} \  \text{if} \ n_{p+1,2}=n_{p,2}.\]
If we continue in this way, ``removing'' one by one twisted quasi-particles from the monomial $b$ and by checking in each step that monomial vectors are in the set $\mathcal{B}$, after finitely many steps we arrive at $v_L=0$, which is a contradiction.

Now, consider a linear linear combination of elements from $\mathcal{B}$ satisfying
\begin{equation}\label{p6}
\sum_{a\in A}c_ab_av_L=0,
\end{equation}
where $b_a \in B$ are monomials of the same color-type and $c_a \in \mathbb{C}$. We will assume that the monomial $b$ of charge-type $\mathcal{R}'$ and dual-charge-type $\mathcal{R}$ is the smallest monomial in (\ref{p6}) with respect to our linear ordering. Recall that the dual-charge-type $\mathcal{R}$ determines  the projection $\pi_{\mathcal{R}}$ and note that from the definition of the projection it follows that all monomial vectors $b_av_L$  in (\ref{p6}), with charge-type $\mathcal{R}'_{a}$ such that $\mathcal{R}'_{a}> \mathcal{R}' $ (see (\ref{lin1})), will be annihilated. So, after applying  $\pi_{\mathcal{R}}$ to (\ref{p6}), we will have
\begin{equation}\label{p7}
\sum_{a\in A}c_a\pi_{\mathcal{R}}b_av_L=0,
\end{equation}
where all monomial vectors are of the same color-charge-type. Now we apply the above-described procedure of using the maps $1\otimes\cdots \otimes  \Delta_c^T(\lambda_1,-z)^{d} \otimes  1 \otimes \cdots \otimes 1$ and $1\otimes\cdots \otimes 1 \otimes  e_{\a_1} \otimes \cdots \otimes e_{\a_1}$ to the smallest element $b$. During this procedure, all monomial vectors $b_a$ such that $b_a>b$ (see (\ref{lin1})) will be annihilated. From (\ref{p7}) now we have $\pi_{\mathcal{R}}c_abv_L=0$, which then implies $c_a=0$. Repeating this procedure, after finitely many steps we will get that all coefficients $c_a$ of (\ref{p6}) are zero, which proves the theorem.
\end{proof}

\section{Characters of principal subspaces}
We define character of principal subspace $W_L^T$ by
\[\text{ch} W_{L_k}^T= \sum_{m,r_1,\ldots, r_l\geq 0} 
\text{dim} \ {W_{L_k}^T}_{(m,r_1,\ldots, r_l)}q^{m}y^{r_1}_{1}\cdots y^{r_l}_{l},\]
where ${W_{L_k}^T}_{(m,r_1,\ldots, r_l)}$ is a weight subspace spanned by monomial vectors of
weight $-m$ and color-type $(r_1,\ldots, r_l)$. 

By rewriting conditions on energies of twisted quasi-particles of a base $\mathcal{B}$ in
terms of the dual-charge-type (and the corresponding charge-type) 
for $\nu \alpha_i=\alpha_i$ we have 
\begin{align*} 
\sum_{p=1}^{r^{(1)}_{i}}\sum_{q=1}^{r^{(1)}_{i-1}}\mathrm{min}\{n_{p,i},n_{q,{i-1}}\}&=\sum_{s=1}^{k}r^{(s)}_{i-1}r_{i}^{(s)},&\\
\nonumber
\sum_{p=1}^{r_{i}^{(1)}} (\sum_{p>p'>0}2\mathrm{min} \{ n_{p,i},
n_{p',i}\}+n_{p,i})&= \sum_{s=1}^{k}r^{(s)^{2}}_{i}, &
\end{align*}
for $\nu \alpha_i \neq\alpha_i$, $v=2$, we have 
\begin{align*} 
\sum_{p=1}^{r^{(1)}_{i}}\frac{1}{2}\sum_{q=1}^{r^{(1)}_{i-1}}\mathrm{min}\{n_{p,i},n_{q,{i-1}}\}&=\frac{1}{2}\sum_{s=1}^{k}r^{(s)}_{i-1}r_{i}^{(s)},&\\
\nonumber
\sum_{p=1}^{r_{i}^{(1)}} (\sum_{p>p'>0}\mathrm{min} \{ n_{p,i},
n_{p',i}\}+\frac{1}{2}n_{p,i})&=\frac{1}{2} \sum_{s=1}^{k}r^{(s)^{2}}_{i}, & 
\end{align*}
and for $\nu \alpha_i \neq\alpha_i$, $v=3$
\begin{align*} 
\sum_{p=1}^{r_{i}^{(1)}} (\frac{2}{3}\sum_{p>p'>0}\mathrm{min} \{ n_{p,i},
n_{p',i}\}+\frac{1}{3}n_{p,i})&=\frac{1}{3} \sum_{s=1}^{k}r^{(s)^{2}}_{i}. &
\end{align*}
Now, we have:
\begin{thm}
For each of the affine Lie algebras $A_{2l-1}^{(2)}, D_l^{(2)}, E_6^{(2)},$ and $D_4^{(3)}$, the principal subspace $W^T_{L_k}$ of $L^{\hnu}(k\Lambda_0)$ has multigraded dimension given by:
\begin{itemize}
\item for $A_{2l-1}$:
\begin{align*} 
&\mathrm{ch} \  W^T_{L_k}&\\
\nonumber
= &\sum_{\substack{r_{1}^{(1)}\geq \cdots\geq r_{1}^{(k)}\geq 0\\ \cdots \\r_{l}^{(1)}\geq \cdots\geq r_{l}^{(k)}\geq 
0}}
\frac{q^{\frac{1}{2}\sum_{i=1}^{l-1}\sum_{s=1}^{k}r_i^{(s)^2}+\sum_{s=1}^{k}r_l^{(s)^2}-\frac{1}{2}\sum_{i=2}^{l-1}\sum_{s=1}^{k}r_{i-1}^{(s)}r_{i}^{(s)}-\sum_{s=1}^{k}r_{l-1}^{(s)}r_{l}^{(s)}}}
{\prod_{i=1}^{l-1}\left((q^{\frac{1}{2}};q^{\frac{1}{2}})_{r^{(1)}_{i}-r^{(2)}_{i}}\cdots (q^{\frac{1}{2}};q^{\frac{1}{2}})_{r^{(k)}_{i}}\right) (q)_{r^{(1)}_{l}-r^{(2)}_{l}}\cdots (q)_{r^{(k)}_{l}}}y^{r^{(1)}_1+\cdots +r^{(k)}_1}_{1} \cdots y^{r^{(1)}_l+\cdots +r^{(k)}_l}_{l}&
\end{align*}

\item  for $D_l$ when $v=2$:

\begin{align*} 
&\mathrm{ch} \  W^T_{L_k}&\\
\nonumber
= &\sum_{\substack{r_{1}^{(1)}\geq \cdots\geq r_{1}^{(k)}\geq 0\\ \cdots \\r_{l-1}^{(1)}\geq \cdots\geq r_{l-1}^{(k)}\geq 
0}}
\frac{q^{\sum_{i=1}^{l-2}\sum_{s=1}^{k}r_i^{(s)^2}+\frac{1}{2}\sum_{s=1}^{k}r_{l-1}^{(s)^2}-\sum_{i=2}^{l-1}\sum_{s=1}^{k}r_{i-1}^{(s)}r_{i}^{(s)}}}
{\prod_{i=1}^{l-2}\left((q)_{r^{(1)}_{i}-r^{(2)}_{i}}\cdots  (q)_{r^{(k)}_{i}}\right)(q^{\frac{1}{2}};q^{\frac{1}{2}})_{r^{(1)}_{l-1}-r^{(2)}_{l-1}}\cdots (q^{\frac{1}{2}};q^{\frac{1}{2}})_{r^{(k)}_{l-1}}}y^{r^{(1)}_1+\cdots +r^{(k)}_1}_{1} \cdots y^{r^{(1)}_{l-1}+\cdots +r^{(k)}_{l-1}}_{l-1}&
\end{align*}

\item  for $E_6$:

\begin{align*} 
&\mathrm{ch} \  W^T_{L_k}&\\
\nonumber
= &\sum_{\substack{r_{1}^{(1)}\geq \cdots\geq r_{1}^{(k)}\geq 0\\\substack{ \ldots \\ r_{4}^{(1)}\geq \cdots\geq r_{4}^{(k)}\geq 
0}}}
\frac{q^{\frac{1}{2}\sum_{i=1}^{2}\sum_{s=1}^{k}r_i^{(s)^2}+\sum_{i=3}^{4}\sum_{s=1}^{k}r_i^{(s)^2}-\sum_{s=1}^{k}(r_{1}^{(s)}r_{2}^{(s)}+r_{2}^{(s)}r_{3}^{(s)}+r_{3}^{(s)}r_{4}^{(s)})}}
{\prod_{i=1,2}\left((q^{\frac{1}{2}};q^{\frac{1}{2}})_{r^{(1)}_{i}-r^{(2)}_{i}}\cdots (q^{\frac{1}{2}};q^{\frac{1}{2}})_{r^{(k)}_{i}}\right)\left(\prod_{i=3,4} (q)_{r^{(1)}_{i}-r^{(2)}_{i}}\cdots (q)_{r^{(k)}_{i}}\right)}y^{r^{(1)}_1+\cdots +r^{(k)}_1}_{1} \cdots y^{r^{(1)}_4+\cdots +r^{(k)}_4}_{4}&
\end{align*}

\item for $D_4$ when $v=3$:

\begin{align*} 
&\mathrm{ch} \  W^T_{L_k}&\\
\nonumber
= &\sum_{\substack{r_{1}^{(1)}\geq \cdots\geq r_{1}^{(k)}\geq 0\\\substack{  r_{2}^{(1)}\geq \cdots\geq r_{2}^{(k)}\geq 
0}}}
\frac{q^{\frac{1}{3}\sum_{s=1}^{k}r_1^{(s)^2}+ \sum_{s=1}^{k}r_2^{(s)^2}-\sum_{s=1}^{k}r_{1}^{(s)}r_{2}^{(s)}}}
{(q^{\frac{1}{3}};q^{\frac{1}{3}})_{r^{(1)}_{1}-r^{(2)}_{1}}\cdots (q^{\frac{1}{3}};q^{\frac{1}{3}})_{r^{(k)}_{1}}  (q)_{r^{(1)}_{2}-r^{(2)}_{2}}\cdots (q)_{r^{(k)}_{2}}}y^{r^{(1)}_1+\cdots +r^{(k)}_1}_{1}  y^{r^{(1)}_2+\cdots +r^{(k)}_2}_{2}.&
\end{align*}
\end{itemize}
\end{thm}

\section*{Acknowledgement}
We are grateful to Mirko Primc for his useful comments and support. The first named author would also like to thank Dra\v zen Adamovi\' c. The first author is partially supported by the Croatian Science Foundation under the project 2634 and by the QuantiXLie Centre of Excellence, a project cofinanced by the Croatian Government and European Union through the European Regional Development Fund - the Competitiveness and Cohesion Operational Programme  (Grant KK.01.1.1.01.0004).

\vspace{.3in}

\vspace{.2in}

\noindent{\small \sc Department of Mathematics, University of Rijeka, Radmile Matej\v{c}i\'{c} 2, 51 000 Rijeka, Croatia
} \\ {\em E--mail address}: mbutorac@math.uniri.hr

\vspace{.2in}
\noindent{\small \sc Department of Mathematics and Computer Science, Ursinus College, 
Collegeville, PA 19426} \\ {\em E--mail address}: csadowski@ursinus.edu

\end{document}